\documentclass[a4paper]{amsart}

\title{The topological Singer construction}
\author{Sverre Lun{\o}e--Nielsen}
\address{Department of Mathematics, University of Oslo, Norway}
\email{sverreln@math.uio.no} \urladdr{http://folk.uio.no/sverreln}
\author{John Rognes}
\address{Department of Mathematics, University of Oslo, Norway}
\email{rognes@math.uio.no} \urladdr{http://folk.uio.no/rognes}
\date{October 27th 2010}

\usepackage{amsmath,amssymb,mathrsfs,amsxtra,amsrefs}
\usepackage{fullpage}
\usepackage[all]{xy}
\newdir{ >}{{}*!/-5pt/@{>}} %% cf xyguide exercise 14
\SelectTips{cm}{} %% cf xyguide page 9

\newtheorem{thm}{Theorem}[section]
\newtheorem{lemma}[thm]{Lemma}
\newtheorem{prop}[thm]{Proposition}
\newtheorem{cor}[thm]{Corollary}
\theoremstyle{definition}
\newtheorem{dfn}[thm]{Definition}
\theoremstyle{remark}
\newtheorem{remark}[thm]{Remark}
\numberwithin{equation}{section}

\renewcommand{\:}{\colon}

\newcommand{\A}{\mathscr{A}}
\newcommand{\C}{\mathbb{C}}
\newcommand{\F}{\mathbb{F}}
\newcommand{\N}{\mathbb{N}}
\newcommand{\R}{\mathbb{R}}
\newcommand{\T}{\mathbb{T}}
\newcommand{\Z}{\mathbb{Z}}

\newcommand{\bfn}{\mathbf{n}}
\newcommand{\fA}{\mathfrak{A}}

\newcommand{\into}{\hookrightarrow}
\newcommand{\longto}{\longrightarrow}

\newcommand{\U}{\mathscr{U}}
\renewcommand{\S}{\mathscr{S}}

\DeclareMathOperator{\im}{im}
\DeclareMathOperator{\ad}{ad}
\DeclareMathOperator{\alg}{alg}
\DeclareMathOperator{\sd}{sd}
\DeclareMathOperator{\pre}{pre}

\DeclareMathOperator{\Tor}{Tor}
\DeclareMathOperator{\Ext}{Ext}
\DeclareMathOperator{\Map}{Map}
\newcommand{\map}{F}

\DeclareMathOperator*{\Rlim}{Rlim}
\DeclareMathOperator*{\colim}{colim}
\DeclareMathOperator*{\hocolim}{hocolim}
\DeclareMathOperator*{\holim}{holim}

\DeclareMathOperator{\Hom}{Hom}
\DeclareMathOperator{\Sq}{Sq}
\DeclareMathOperator{\SP}{P}
\DeclareMathOperator{\THH}{THH}

\newcommand{\tEG}{\widetilde{EG}}
\newcommand{\tEC}{\widetilde{EC}}
\newcommand{\tF}[1]{\widetilde{E}_{#1}}

\newcommand{\tH}{\widehat{H}}
\newcommand{\hatE}{\widehat{E}}
\newcommand{\ctensor}{\mathbin{\widehat{\otimes}}}

\begin{document}

\begin{abstract}
We study the continuous (co-)homology of towers of spectra, with emphasis
on a tower with homotopy inverse limit the Tate construction $X^{tG}$
on a $G$-spectrum $X$.   When $G=C_p$ is cyclic of prime order and
$X=B^{\wedge p}$ is the $p$-th smash power of a bounded below spectrum $B$
with $H_*(B; \F_p)$ of finite type, we prove that $(B^{\wedge p})^{tC_p}$
is a topological model for the Singer construction $R_+(H^*(B; \F_p))$ on
$H^*(B; \F_p)$.  There is a map $\epsilon_B \: B\to (B^{\wedge p})^{tC_p}$
inducing the $\Ext_\A$-equivalence $\epsilon\: R_+(H^*(B; \F_p))\to
H^*(B; \F_p)$.  Hence $\epsilon_B$ and the canonical map $\Gamma \:
(B^{\wedge p})^{C_p}\to (B^{\wedge p})^{hC_p}$ are $p$-adic equivalences.
\end{abstract}

\maketitle{}
\setcounter{tocdepth}{1}
\tableofcontents{}

\section{Introduction}
Let $p$ be a fixed prime.  We study homology and cohomology groups with
$\F_p$-coefficients associated to towers of spectra
\begin{equation} \label{equ_introtower}
Y = \holim_{n\to-\infty} Y_n \to \cdots \to Y_{n-1} \to Y_n \to \cdots \,,
\end{equation}
where each $Y_n$ is bounded below and of finite type over~$\F_p$, and
$Y$ is equal to the homotopy inverse limit of the tower.  By a result
of Caruso, May and Priddy \cite{CMP1987}, there exists an \emph{inverse
limit of Adams spectral sequences} that calculates the homotopy groups
of the $p$-completion $Y\sphat_p = \map(S^{-1}\!/p^\infty, Y)$ of~$Y$,
where $\map(-, -)$ denotes the function spectrum and $S^{-1}\!/p^\infty$
is the Moore spectrum with homology $\Z/p^\infty$ in degree~$-1$.
The $E_2$-term for this spectral sequence is given by the $\Ext$-groups
of the direct limit of cohomology groups
\begin{equation} \label{equ_introcolim}
H_c^*(Y; \F_p) = \colim_{n\to-\infty} H^*(Y_n; \F_p) \,,
\end{equation}
arising from the tower~\eqref{equ_introtower}, considered as a module
over the mod~$p$ Steenrod algebra $\A$.  We shall refer to this colimit
as the \emph{continuous cohomology groups} of $Y$.

\subsection{The algebraic Singer construction}
A natural question is: how well do we understand the structure of
\eqref{equ_introcolim} as an $\A$-module?  There is an interesting example
of a tower of spectra where this question has the answer: very well.
In fact, this question appeared in the study of Segal's Burnside ring
conjecture for cyclic groups of prime order.  At the heart of W.~H.~Lin's
proof of the case $p=2$, published in \cite{LDMA1980},  lies a careful
study of the $\A$-module
$$
P = H_c^*(\R P^\infty_{-\infty}; \F_2) = 
  \colim_{m\to\infty} H^*(\R P^\infty_{-m}; \F_2) \,,
$$
and its associated $\Ext$-groups $\Ext^{*,*}_\A(P, \F_2)$.  It turns out
that $P = P(x, x^{-1}) = \F_2[x, x^{-1}]$ is isomorphic to the so-called
\emph{Singer construction} $R_+(M)$ on the trivial $\A$-module $M=\F_2$,
up to a degree shift.  The Singer construction has an explicit description
as a module over $\A$.  More importantly, it has the property that
there is a natural $\A$-module homomorphism $\epsilon \: R_+(M) \to M$
that induces an isomorphism
$$
\epsilon^* \: \Ext^{*,*}_\A(M, \F_2) \overset\cong\to
\Ext^{*,*}_\A(R_+(M), \F_2) \,.
$$

\subsection{The topological Singer construction}
Our objective is to present a topological realization and a useful
generalization of these results.  Specifically, for a bounded below
spectrum $B$ of finite type over~$\F_p$, we construct a tower
of spectra
$$
(B^{\wedge p})^{tC_p} = \holim_{n\to-\infty} (B^{\wedge p})^{tC_p}[n]
\to \cdots \to (B^{\wedge p})^{tC_p}[n{-}1]
\to (B^{\wedge p})^{tC_p}[n] \to \cdots
$$
as in~\eqref{equ_introtower}.  Here $C_p$ is the cyclic group of
order~$p$, and $B^{\wedge p}$ is a $C_p$-equivariant model of the $p$-th
smash power of $B$.  We call
$R_+(B) = (B^{\wedge p})^{tC_p}$
the \emph{topological Singer construction} on $B$, and prove in
Theorem~\ref{thm_topmodelsalg} that there is a natural isomorphism
$$
R_+(H^*(B; \F_p)) \cong H_c^*(R_+(B); \F_p)
	= \colim_{n\to-\infty} H^*((B^{\wedge p})^{tC_p}[n]; \F_p)
$$
of $\A$-modules.  Furthermore, we define a natural stable
map $\epsilon_B \: B \to (B^{\wedge p})^{tC_p}$, and prove in
Proposition~\ref{prop_epsilonB} that it realizes the $\A$-module
homomorphism $\epsilon \: R_+(H^*(B; \F_p)) \to H^*(B; \F_p)$ in
continuous cohomology.

The Segal conjecture for $C_p$ follows as a special case of this, when
$B=S$ is the sphere spectrum, since $S^{\wedge p}$ is a model for the
genuinely $C_p$-equivariant sphere spectrum.  More generally, we prove
in Theorem~\ref{thm_fixedpoints} that for any bounded below spectrum
$B$ with $H_*(B; \F_p)$ of finite type the canonical map
$$
\Gamma \: (B^{\wedge p})^{C_p} \longto (B^{\wedge p})^{hC_p} \,,
$$
relating the $C_p$-fixed points and $C_p$-homotopy fixed points for
$B^{\wedge p}$, becomes a homotopy equivalence after $p$-completion.
In \cite{BBLR}*{1.7}, with M.~B{\"o}kstedt and R.~Bruner, we
deduce from this that there are $p$-adic equivalences $\Gamma_n \:
(B^{\wedge p^n})^{C_{p^n}} \longto (B^{\wedge p^n})^{hC_{p^n}}$ for
all $n\ge1$.

\subsection{Outline of the paper}
In \textsection \ref{sec_limitsofspectra}, we discuss towers of spectra
as above, and the associated limit systems obtained by applying homology
or cohomology with $\F_p$-coefficients.  When dealing with towers of ring
spectra it is convenient to work in homology, while for the formation of
the $\Ext$-groups mentioned above it is convenient to work in cohomology.
In order to be able to switch back and forth between cohomology and
homology we discuss linear topologies arising from filtrations, and
continuous dualization.  Then we see how an $\A$-module structure in
cohomology dualizes to an $\A_*$-comodule structure in a suitably
completed sense.

In \textsection \ref{sec_algsinger} we recall the algebraic Singer
construction on an $\A$-module $M$, in its cohomological form,
and study a dual homological version, defined for $\A_*$-comodules $M_*$
that are bounded below and of finite type.  We give the details for
the form of the Singer construction that is related to the
group $C_p$, since most references only consider the smaller version
related to the symmetric group $\Sigma_p$.

Then, in \textsection \ref{sec_tate}, we define a specific tower of
spectra $\{ X^{tG}[n] \}_n$ with homotopy inverse limit equivalent
to the Tate construction $X^{tG} = [\tEG \wedge \map(EG_+, X)]^G$ on a
$G$-spectrum $X$.  We consider the associated (co-)homological Tate
spectral sequences, and compare our approach of working with homology
groups to earlier papers that focused directly on homotopy groups.

In \textsection \ref{sec_singer}, we specialize to the case when
$X = B^{\wedge p}$.  This is also where we discuss the genuinely
$C_p$-equivariant model of the spectrum $B^{\wedge p}$, given by
the $p$-fold smash product of (symmetric) spectra introduced by
B{\"o}kstedt \cite{B1} in his definition of topological
Hochschild homology.  It is for this particular $C_p$-equivariant model
that we can define the natural stable map $\epsilon_B \: B \to R_+(B)$
realizing Singer's homomorphism $\epsilon$.

\subsection{Notation}
Spectra will usually be named $B$, $X$ or $Y$.  Here $B$ will be
a bounded below spectrum or $S$-algebra of finite type over~$\F_p$.
Spectra denoted by $X$ will be equipped with an equivariant structure.
The main examples we have in mind are the $p$-fold smash product $X =
B^{\wedge p}$ treated here, and the topological Hochschild homology
spectrum $X = \THH(B)$ treated in the sequel~\cite{LR2}.  When dealing
with generic towers of spectra, we will use $Y$.  The example of main
interest is the Tate construction $Y = X^{tG}$ on some $G$-equivariant
spectrum $X$.

We write $\A$ for the mod~$p$ Steenrod algebra and $\A_*$ for its
$\F_p$-linear dual.  We will work with left modules over $\A$ and left
comodules under $\A_*$.  In the body of the paper we write $H_*(B) =
H_*(B; \F_p)$ and $H^*(B) = H^*(B; \F_p)$, for brevity.  Unlabeled $\Hom$
means $\Hom_{\F_p}$, and $\otimes$ means $\otimes_{\F_p}$.

\subsection{History and notation of the Singer construction}
The Singer construction appeared originally for $p=2$ in \cite{S1980} \and
\cite{S1981}, and for $p$ odd in \cite{LS1982}.  The work presented here
concentrates on its relation to the calculations by Lin and Gunawardena
and their work on the Segal conjecture for groups of prime order.
A published account for the case of the group of order $2$ is found
in \cite{LDMA1980}.  A further study appears in \cite{AGM1985}, where
a more conceptual definition of the Singer construction is given.

In W.~Singer's paper \cite{S1980}, the following problem is posed:
Let $M$ be an unstable $\A$-module and let $\Delta = \F_2\{\Sq^r \mid
r\in \Z\}$.  There is a map of graded $\F_2$-vector spaces
$$
  d\:\Delta\otimes M\to M
$$
taking $\Sq^r\otimes m$ to $\Sq^r(m)$ for $r\ge0$, and to $0$ for $r<0$.
Does there exist a natural $\A$-module structure on the source of $d$
rendering this map an $\A$-linear homomorphism?  Singer answers this
question affirmatively, by using an idea of Wilkerson \cite{W1977}
to construct the $\A$-module that he denotes $R_+(M)$, an $\A$-module
map $d \: R_+(M) \to M$ of degree~$1$, and an isomorphism $R_+(M) \cong
\Delta \otimes M$, also of degree~$1$, that makes the two maps called
$d$ correspond.  In the end, the construction does not depend on $M$
being unstable.

In Li and Singer's paper \cite{LS1982}, the odd-primary version of
this problem is solved, with $\Delta = \F_p\{ \beta^i \SP^r \mid i
\in \{0,1\}, r \in \Z\}$.  Starting with that paper there is a degree
shift in the notation: $R_+(M)$ now denotes the suspension of $R_+(M)$
from Singer's original paper, so that the $\A$-module map $d \: R_+(M)
\to M$ is of degree~$0$.

In connection with the Segal conjecture, Adams, Gunawardena and
Miller \cite{AGM1985} published an algebraic account of the Singer
construction, for all primes $p$.  They write $T'(M)$ for the $\A$-module
denoted $R_+(M)$ in \cite{S1980} and \cite{S1981}, and let $T''(M) =
\Sigma T'(M)$ be its suspension, denoted $R_+(M)$ in \cite{LS1982}.
For the trivial $\A$-module $\F_p$, $T''(\F_p)$ is isomorphic to the
Tate homology $\tH_{-*}(\Sigma_p; \F_p)$, which can be obtained by
localization from $\Sigma H^*(\Sigma_p; \F_p)$.  Hence $T'(\F_p) =
\tH^*(\Sigma_p; \F_p)$ is a localized form of $H^*(\Sigma_p; \F_p)$.
Adams, Gunawardena and Miller are not only concerned with the Segal
conjecture for the groups of prime order, but also for the elementary
abelian groups $(C_p)^n$, so the cross product in cohomology is important
for them.  They therefore prefer $T'$ over $T''$.  In fact, they really
work with an extended functor $T(M) = T(\F_p) \otimes_{T'(\F_p)} T'(M)$,
where $T(\F_p) = \tH^*(C_p; \F_p)$.

In our context, for the cohomological study of towers of ring spectra
it will be the coproduct in Tate homology that is most important, which
is why we prefer $T''$ over $T'$.  Again, our emphasis is on $C_p$
instead of $\Sigma_p$, so that the Singer-type functor we shall
work with is the extension
$$
\tH_{-*}(C_p; \F_p) \otimes_{T''(\F_p)} T''(M)
$$
of $T''(M)$, which is $(p-1)$ times larger than the $T''(M) = R_+(M)$
of Li and Singer.  This is the functor we shall denote $R_+(M)$, so
that $R_+(\F_p) = \tH_{-*}(C_p; \F_p)$, and $R_+(M) = \Sigma T(M)$ in
the notation of \cite{AGM1985}.

The connection between the Singer construction and the
continuous cohomology of a tower of spectra, displayed below as
\eqref{equ_inverseextendedpowers} for $G=\Sigma_p$, was found by Haynes
Miller, and explained in \cite{BMMS}*{II.5.1}.  There the functor denoted
$R_+$ is the same as in \cite{LS1982} for odd $p$, shifted up one degree
from \cite{S1981} for $p=2$.

In the following we will make use of the fact that the Singer construction
on an $\A$-module $M$ comes equipped with the homomorphism of $\A$-modules
$\epsilon\: R_+(M)\to M$.  In Singer's work \cite{S1981}, this map has
degree $+1$ (and was named $d$), whereas in \cite{LS1982} and \cite{BMMS}
it has degree 0.  We choose to follow the latter conventions, because this
is the functor that with no shift of degrees describes our continuous
(co-)homology groups.  Furthermore, the homomorphism $\epsilon$ will be
realized by a map of spectra, and will therefore be of degree zero.

We choose to write $R_+(M)$ instead of $T(M)$ or any of its variants,
because the letter $T$ is heavily overloaded by the presence of $\THH$,
the Tate construction and the circle group $\T$.  To add to the confusion,
the letter $T$ is also used in Singer's \cite{S1981} work, but with a
different meaning than the $T$ appearing in \cite{AGM1985}.

\subsection{Acknowledgments}
This work started out as a part of the first author's PhD thesis at the
University of Oslo, supervised by the second author.  The first author
wishes to express the deepest gratitude to Prof.~Rognes for offering ideas,
help and support during the work of his thesis.  Thanks are also due to
R.~Bruner, G.~Carlsson, I.~C.~Borge, S.~Galatius, T.~Kro, M.~B{\"o}kstedt,
M.~Brun and B.~I.~Dundas for interest, comments and discussions.

\section{Limits of spectra} \label{sec_limitsofspectra}
We introduce our conventions regarding towers of spectra and their
associated (co-)homology groups.  Our motivation is the result of
Caruso, May and Priddy, saying that that there is an inverse limit of
Adams spectral sequences arising from such towers.  The input for this
inverse limit of Adams spectral sequences will give us the definition
of continuous (co-)homology groups.

\subsection{Inverse limits of Adams spectral sequences}

\begin{dfn}
Let $R$ be a (Noetherian) ring.  A graded $R$-module $M_*$ is
\emph{bounded below} if there is an integer $\ell$ such that $M_* =
0$ for all $* < \ell$.  It is of \emph{finite type} if it is finitely
generated over $R$ in each degree.

A spectrum $B$ is \emph{bounded below} if its homotopy $\pi_*(B)$
is bounded below as a graded abelian group.  It is of \emph{finite
type over~$\F_p$} if its mod~$p$ homology $H_*(B) = H_*(B; \F_p)$
is of finite type as a graded $\F_p$-vector space.  The spectrum $B$
is of \emph{finite type over $\widehat\Z_p$} if its homotopy $\pi_*(B)$
is of finite type as a graded $\widehat\Z_p$-module.
\end{dfn}

Let $\{Y_n\}_{n\in\Z}$ be a sequence of spectra, with maps $f_n \: Y_{n-1}
\to Y_n$ for all integers~$n$.  Assume that each $Y_n$ is bounded below
and of finite type over~$\F_p$, and let $Y$ be the homotopy inverse
limit of this system:
\begin{equation} \label{equ_tower}
  Y \longto \cdots \longto Y_{n-1} \overset{f_n}\longto Y_n
  \longto \cdots
\end{equation}
In general, $Y$ will neither be bounded below nor of finite type
over~$\F_p$.

For each $n$ there is an Adams spectral sequence $\{E^{*,*}_r(Y_n)\}_r$
with $E_2$-term
$$
E_2^{s,t}(Y_n) = \Ext_{\A}^{s,t}(H^*(Y_n), \F_p)
	\Longrightarrow \pi_{t-s}((Y_n)\sphat_p) \,,
$$
converging strongly to the homotopy groups of the $p$-completion of $Y_n$.
Each map in the tower~\eqref{equ_tower} induces a map of Adams spectral
sequences $f_n \: \{E^{*,*}_r(Y_{n-1})\}_r \to \{E^{*,*}_r(Y_n)\}_r$.
For every $r$ let
$$
E^{*,*}_r(\underline{Y}) = \lim_{n\to-\infty} E^{*,*}_r(Y_n) \,,
$$
and similarly for the $d_r$-differentials.  We now state and prove a
slightly sharper version of \cite{CMP1987}*{7.1}.

\begin{prop} \label{prop_CMPss}
Let $\{Y_n\}_n$ be a tower of spectra such that each $Y_n$ is bounded
below and of finite type over~$\F_p$, and let $Y = \holim_n Y_n$.
Then the bigraded groups $\{E^{*,*}_r(\underline{Y})\}_r$ are the terms
of a spectral sequence, with $E_2$-term
$$
E^{s,t}_2(\underline{Y})
        \cong \Ext_{\A}^{s,t}(\colim_{n\to-\infty} H^*(Y_n), \F_p)
\Longrightarrow \pi_{t-s}(Y\sphat_p) \,,
$$
converging strongly to the homotopy groups of the $p$-completion of $Y$.
\end{prop}

The difference between this statement and the statement in \cite{CMP1987}
lies in the hypothesis on the $Y_n$: we do not assume that $Y_n$ is
$p$-complete, and weaken the condition that $Y_n$ should be of finite
type over $\widehat\Z_p$ to the condition that $H_*(Y_n; \F_p)$ should
be of finite type.
We refer to the spectral sequence $\{E^{*,*}_r(\underline{Y})\}_r$ as
the \emph{inverse limit of Adams spectral sequences} associated to the
tower $\{Y_n\}_n$.

\begin{proof}
For any bounded below spectrum $B$ of finite type over~$\F_p$, the Adams
spectral sequence converges strongly to $\pi_*(B\sphat_p)$.  Since the
$E_2$- and $E_\infty$-terms of this spectral sequence are of finite type,
the abelian groups $\pi_*(B\sphat_p)$ are compact and Hausdorff in
the topology given by the Adams filtration.

The category of compact Hausdorff abelian groups is an abelian category,
as is the category of discrete abelian groups.  The Pontryagin duality
functor assigns to each abelian group $G$ its character group $\Hom(G,
\T)$, where $\T = S^1$ is the circle group.  It induces a contravariant
equivalence between the category of compact Hausdorff abelian groups and
the category of discrete abelian groups.  The functor taking a filtered
diagram of discrete abelian groups to its colimit is well-known to be
exact.  It follows that the functor taking a filtered diagram of compact
Hausdorff abelian groups to its inverse limit is also an exact functor.
In particular, passing to filtered inverse limits commutes with the
formation of kernels, images, cokernels and homology, in the abelian
category of compact Hausdorff abelian groups.

We now adapt the proof of \cite{CMP1987}*{7.1}, using this version of
the exactness of the inverse limit functor.  First, we construct a double
tower diagram of spectra
\begin{equation} \label{equ_towAdamsres}
\xymatrix{
\dots \ar[r] & Z_s \ar[r] \ar[d] & \dots \ar[r] & Z_0 = Y \ar[d] \\
& \vdots \ar[d] & & \vdots \ar[d] \\
\dots \ar[r] & Z_{n-1,s} \ar[r] \ar[d] & \dots \ar[r] & Z_{n-1,0} = Y_{n-1} \ar[d]^{f_n} \\
\dots \ar[r] & Z_{n,s} \ar[r] & \dots \ar[r] & Z_{n,0} = Y_n
}
\end{equation}
where the $n$-th row is an Adams resolution
of $Y_n$, such that each $Z_{n,s}$ is a bounded below spectrum of finite
type over~$\F_p$.  The $n$-th row can be obtained in a functorial way by
smashing $Y_n$ with a fixed Adams resolution for the sphere spectrum $S$.
The top row consists of the homotopy limits $Z_s = \holim_n Z_{n,s}$
for all $s\ge0$.  By assumption, the homotopy cofiber $K_{n,s}$ of the
map $Z_{n,s+1} \to Z_{n,s}$ is a wedge sum of suspended copies of the
Eilenberg--Mac\,Lane spectrum $H = H\F_p$, also bounded below and of finite
type over~$\F_p$.  The exact couple
$$
\xymatrix{
\pi_*(Z_{n,s+1}) \ar[r]^i & \pi_*(Z_{n,s}) \ar[d]^j \\
& \pi_*(K_{n,s}) \ar@{-->}[ul]^k
}
$$
generates the Adams spectral sequence $\{E_r^{*,*}(Y_n)\}_r$, with
$E_1^{s,t}(Y_n) = \pi_{t-s}(K_{n,s})$.  The dashed arrow has degree~$-1$.

Now consider the $p$-completion of diagram~\eqref{equ_towAdamsres}.
The $n$-th row becomes
\begin{equation} \label{equ_nthrowp}
\dots \to (Z_{n,s})\sphat_p \to \dots \to (Z_{0,s})\sphat_p =
	(Y_n)\sphat_p
\end{equation}
and the homotopy cofiber of $(Z_{n,s+1})\sphat_p \to (Z_{n,s})\sphat_p$
is the $p$-completion of $K_{n,s}$, which is just $K_{n,s}$ again.
There is therefore a second exact couple
\begin{equation} \label{equ_excplnp}
\xymatrix{
\pi_*((Z_{n,s+1})\sphat_p) \ar[r]^i & \pi_*((Z_{n,s})\sphat_p) \ar[d]^j \\
& \pi_*(K_{n,s}) \ar@{-->}[ul]^k
}
\end{equation}
for each $n$, which generates the same spectral sequence as the first
one.  Furthermore, in the second exact couple all the (abelian) homotopy
groups are compact Hausdorff, since the spectra $Z_{n,s}$ and $K_{n,s}$
are all bounded below and of finite type over~$\F_p$.

The $p$-completion of the top row in~\eqref{equ_towAdamsres}
is the homotopy inverse limit over~$n$ of the $p$-completed
rows~\eqref{equ_nthrowp}.  The exactness of filtered limits in
the category of compact Hausdorff abelian groups now implies
that there are isomorphisms $\pi_*((Z_s)\sphat_p) \cong \lim_n
\pi_*((Z_{n,s})\sphat_p)$ for all $s$.  Furthermore, the inverse limit
over $n$ of the exact couples~\eqref{equ_excplnp} defines a third exact
couple (of compact Hausdorff abelian groups)
\begin{equation} \label{equ_excpllim}
\xymatrix{
\pi_*((Z_{s+1})\sphat_p) \ar[r]^i & \pi_*((Z_s)\sphat_p) \ar[d]^j \\
& \displaystyle{\lim_n} \, \pi_*(K_{n,s}) \ar@{-->}[ul]^k \rlap{\,,}
}
\end{equation}
which generates the spectral sequence $\{E_r^{*,*}(\underline{Y})\}_r$
that we are after.  Here $E_1^{s,t}(\underline{Y}) \cong \lim_n
E_1^{s,t}(Y_n)$, and by induction on $r$ the same isomorphism holds for
each $E_r$-term, since $E_{r+1}^{*,*}$ is the homology of $E_r^{*,*}$
with respect to the $d_r$-differentials, and we have seen that the
formation of these limits commutes with homology.  In particular, each
abelian group $E_r^{s,t}(\underline{Y})$ is compact Hausdorff.

The identification of the $E_2$-term for $s=0$ amounts to the
isomorphism
$$
\lim_n \Hom_{\A}(H^*(Y_n), N) \cong
\Hom_{\A}(\colim_n H^*(Y_n), N)
$$
for $N = \Sigma^t \F_p$.  The general case follows, since we can compute
$\Ext^s_{\A}$ by means of an injective resolution of $\F_p$.

We must now check the convergence of this spectral sequence, which we
(and \cite{CMP1987}) do following Boardman \cite{B1999}.  The Adams
resolution for each $Y_n$ is constructed so that
$$
\lim_s \pi_*((Z_{n,s})\sphat_p) = \Rlim_s \pi_*((Z_{n,s})\sphat_p) = 0
\,.
$$
These two conditions ensure that the Adams spectral sequence for $Y_n$
converges conditionally \cite{B1999}*{5.10}.  The standard interchange of
limits isomorphism gives
$$
\lim_s \pi_*((Z_s)\sphat_p) \cong \lim_n \lim_s \pi_*((Z_{n,s})\sphat_p)
= 0 \,.
$$
Moreover, the exactness of the inverse limit functor in this case implies
that the derived limit
$$
\Rlim_s \pi_*((Z_s)\sphat_p) = 0
$$
vanishes, too.  Hence the inverse limit Adams spectral sequence
generated by~\eqref{equ_excpllim} is conditionally convergent to
$\pi_*(Y\sphat_p)$.  This is a half-plane spectral sequence
with entering differentials, in the sense of Boardman.
For such spectral sequences, strong convergence follows
from conditional convergence together with the vanishing of the groups
$$
RE^{s,t}_\infty = \Rlim_r E_r^{s,t}(\underline{Y}) \,,
$$
see \cite{B1999}*{5.1, 7.1}.  Again, the vanishing of this $\Rlim$ is
ensured by the exactness of $\lim$ for the compact Hausdorff abelian
groups $E_r^{s,t}(\underline{Y})$.
\end{proof}

\subsection{Continuous (co-)homology} \label{sec_duality}
The spectral sequence in Proposition~\ref{prop_CMPss} is central to the
proof of the Segal conjecture for groups of prime order and will be the
foundation for the present work.  Our work will, in analogy with Lin's
proof of the Segal conjecture, focus on the properties
of the $E_2$-term of the above spectral sequence.

\begin{dfn} \label{dfn_continuous}
Let $\{Y_n\}_n$ be a tower of spectra such that each $Y_n$ is bounded
below and of finite type over~$\F_p$, and let $Y = \holim_n Y_n$.
Define the \emph{continuous cohomology} of $Y$ as the colimit
$$
H_c^*(Y) = \colim_{n\to-\infty} H^*(Y_n) \,.
$$
Dually, define the \emph{continuous homology} of $Y$ as the inverse limit
$$
H_*^c(Y) = \lim_{n\to-\infty} H_*(Y_n) \,.
$$
\end{dfn}

Note that we choose to suppress from the notation the tower of which $Y$
is a homotopy inverse limit, even if the continuous cohomology groups
do depend on the choice of inverse system.  For example, let $p=2$
and let $Y=S\sphat_2$ be the $2$-completed sphere spectrum.  Since $Y$
is bounded below and of finite type over $\F_2$, we may express $Y$ by
the constant tower of spectra.  But by W.~H.~Lin's theorem, $S\sphat_2
\simeq \holim_n \Sigma \R P^\infty_n$, where each $\Sigma\R
P^\infty_n$ is also bounded below and of finite type over $\F_2$.
Now $\colim_n H^*(\Sigma \R P^\infty_n) = \Sigma P(x, x^{-1}) = R_+(\F_2)$
is much larger than $H^*(S\sphat_2) = \F_2$.

By the universal coefficient theorem and our finite type assumptions, the
$\F_p$-linear dual of $H^*_c(Y)$ is naturally isomorphic to $H_*^c(Y)$.
The continuous homology of $Y$ will often not be of finite type, so
its dual is in general not isomorphic to the continuous cohomology.
However, if we take into account the linear topology on the inverse limit,
given by the kernel filtration induced from the tower, we do get that the
continuous dual of the continuous homology is isomorphic to the continuous
cohomology.  We discuss this in \textsection \ref{subsec_dualization}.

Note that the continuous cohomology is a direct limit of bounded below
$\A$-modules.  The direct limit might of course not be bounded below,
but we do get a natural $\A$-module structure on $H^*_c(Y)$ in the
category of all $\A$-modules.  Dually, the continuous homology is an
inverse limit of bounded below $\A_*$-comodules, but the inverse limit
might be neither bounded below nor an $\A_*$-comodule in the usual,
algebraic, sense.  Instead we get a completed coaction of $\A_*$
$$
H_*^c(Y) \to \A_* \ctensor H_*^c(Y) \,,
$$
where $\ctensor$ is the tensor product completed with respect to the
above-mentioned linear topology on the continuous homology.  We discuss
this in \textsection \ref{subsec_limitsofcomodules}.

\subsection{Filtrations}
For every $n\in \Z$, let $A^n$ be a graded $\F_p$-vector space and assume
that these vector spaces fit into a sequence
\begin{equation} \label{equ_cohomtower}
  0 \longto \cdots \longto A^n \longto A^{n-1} \longto \cdots \longto
  A^{-\infty}
\end{equation}
with trivial inverse limit, and colimit denoted by $A^{-\infty}$.
We assume further that each $A^n$ is of finite type.  Let $A_n =
\Hom(A^n, \F_p)$ be the dual of $A^n$.  The diagram
above dualizes to a sequence
\begin{equation} \label{equ_homtower}
  A_{-\infty} \longto \cdots \longto A_{n-1} \longto A_n \longto \cdots
  \longto 0
\end{equation}
with inverse limit
$$
A_{-\infty} = \lim_n A_n = \lim_n \Hom(A^n, \F_p)
\cong \Hom(\colim_n A^n, \F_p) = \Hom(A^{-\infty}, \F_p)
$$
isomorphic to the dual of $A^{-\infty}$, and trivial colimit.  The last
fact follows from the assumption that $A^n$ is finite dimensional in
each degree.  Indeed,
$$
  \lim_n A^n \cong \lim_n \Hom(A_n,\F_p) \cong
	\Hom(\colim_n A_n,\F_p)
$$
and thus $\lim_n A^n=0$ implies that $\colim_n A_n$ is trivial, since
the latter injects into its double dual.  Furthermore,
the derived inverse limits $\Rlim_n A^n$ and $\Rlim_n A_n$ are zero,
again because $A^n$ and $A_n$ are degreewise finite.

Adapting Boardman's notation \cite{B1999}*{5.4}, we define filtrations
of the colimit of~\eqref{equ_cohomtower} and the inverse limit of
\eqref{equ_homtower} using the corresponding sequential limit systems.

\begin{dfn} \label{def_filtr111}
For each $n \in \Z$, let
$$
F^nA^{-\infty} = \im(A^{n} \to A^{-\infty})
$$
and
$$
F_nA_{-\infty} = \ker(A_{-\infty} \to A_n) \,.
$$
Then
\begin{equation} \label{equ_cohomfiltration}
  \cdots \subset F^nA^{-\infty} \subset F^{n-1}A^{-\infty} \subset \cdots
  \subset A^{-\infty}
\end{equation}
and
\begin{equation} \label{equ_homfiltration}
   \cdots\subset F_{n-1}A_{-\infty}\subset F_nA_{-\infty}\subset \cdots
   \subset A_{-\infty}
\end{equation}
define a decreasing (resp.~increasing) sequence of subspaces of
$A^{-\infty}$ (resp.~$A_{-\infty}$).
\end{dfn}

The filtration~\eqref{equ_cohomfiltration} clearly exhausts
$A^{-\infty}$.  Since each $A^n$ and $\ker(A^n \to A^{-\infty})$ is
of finite type, the right derived limits $\Rlim_n A^n$ and $\Rlim_n
\ker(A^n \to A^{-\infty})$ are both zero.  By assumption $\lim_n A^n =
0$, hence both $\lim_n F^nA^{-\infty}$ and $\Rlim_n F^nA^{-\infty}$
vanish.  In other words, the filtration~\eqref{equ_cohomfiltration} is
Hausdorff and complete, so that the canonical map $A^{-\infty} \to \lim_n
(A^{-\infty}/F^nA^{-\infty})$ is an isomorphism.  Completeness is equivalent
to saying that Cauchy sequences converge in the linear topology given by
the filtration.  That the filtration is Hausdorff is saying that Cauchy
sequences have unique limits.

For the filtration~\eqref{equ_homfiltration}, the proof of
\cite{B1999}*{5.4(b)}) shows, without any hypotheses, that the
filtration is Hausdorff and complete.  It also shows that the filtration
is exhaustive, since the colimit of~\eqref{equ_homtower} is trivial.
We collect these facts in the following lemma.

\begin{lemma} \label{lem_filtration}
Assume that the inverse limit $\lim_n A^n$ in~\eqref{equ_cohomtower}
is trivial and that each $A^n$ is of finite type.  Then both filtrations
given in Definition~\ref{def_filtr111} (of $\colim_n A^n = A^{-\infty}$
resp.~its dual $A_{-\infty}$) are exhaustive, Hausdorff
and complete.
\end{lemma}

\subsection{Dualization} \label{subsec_dualization}
The dual of the inverse limit $A_{-\infty}$ of
\eqref{equ_homtower} is the double dual of the colimit $A^{-\infty}$
of~\eqref{equ_cohomtower}.  It contains this colimit in a canonical
way, but is often strictly bigger, since $A^{-\infty}$ needs not be of
finite type.  To remedy this, we take into account the linear topology
on the limit induced by the inverse system, and dualize by considering
the continuous $\F_p$-linear dual.

In this topology on $A_{-\infty}$, an open neighborhood basis of the
origin is given by the collection of subspaces $\{F_nA_{-\infty}\}_n$.
A continuous homomorphism $A_{-\infty} \to \F_p$ is thus an $\F_p$-linear
function whose kernel contains $F_nA_{-\infty}$ for some $n$.  The set
of these forms an $\F_p$-vector space $\Hom^c(A_{-\infty}, \F_p)$,
which we call the \emph{continuous dual} of $A_{-\infty}$.

\begin{lemma} \label{lem_contdual}
There is a natural isomorphism
$$
  \Hom(A^{-\infty}, \F_p) \cong A_{-\infty}\,.
$$
Give $A_{-\infty}$ the linear topology induced by the system of
neighborhoods $\{F_nA_{-\infty}\}_n$.  Then there is a natural isomorphism
$$
  \Hom^c(A_{-\infty}, \F_p) \cong A^{-\infty}\,.
$$
\end{lemma}

\begin{proof}
The first isomorphism has already been explained.
For the second, we wish to compute
$$
  \Hom^c(A_{-\infty}, \F_p) \cong
	\colim_n \Hom(A_{-\infty}/F_nA_{-\infty}, \F_p) \,.
$$
The dual of the image $F^nA^{-\infty} = \im(A^n \to A^{-\infty})$
is the image
\begin{equation} \label{equ_fna8dual}
\Hom(F^nA^{-\infty}, \F_p) \cong \im(A_{-\infty} \to A_n)
	\cong A_{-\infty}/F_nA_{-\infty} \,,
\end{equation}
and $F^nA^{-\infty}$ is of finite type, so the canonical homomorphism
$$
F^nA^{-\infty} \overset{\cong}{\to} \Hom(A_{-\infty}/F_nA_{-\infty}, \F_p)
$$
into its double dual is an isomorphism.  Passing to the colimit as $n
\to -\infty$ we get the desired isomorphism, since
$\colim_n F^nA^{-\infty} \cong A^{-\infty}$.
\end{proof}

\subsection{Limits of $\A_*$-comodules} \label{subsec_limitsofcomodules}
Until now, the objects of our discussion have been graded vector
spaces over~$\F_p$.  We will now add more structure, and assume that
\eqref{equ_cohomtower} is a diagram of modules over the Steenrod
algebra $\A$.  It follows that the finite terms $A_n$ in the dual tower
\eqref{equ_homtower} are comodules under the dual Steenrod algebra~$\A_*$.
We need to discuss in what sense these comodule structures carry over
to the inverse limit $A_{-\infty}$.

Let $M_*$ be a graded vector space, with a linear topology given by a
system $\{U_\alpha\}_\alpha$ of open neighborhoods, with each $U_\alpha$ a
graded subspace of $M_*$.  We say that $M_*$ is \emph{complete Hausdorff}
if the canonical homomorphism $M_* \overset{\cong}\to \lim_\alpha
(M_*/U_\alpha)$ is an isomorphism.
Let $V_*$ be a graded vector space, bounded below and given the discrete
topology.  By the \emph{completed tensor product} $V_* \ctensor M_*$
we mean the limit $\lim_\alpha (V_* \otimes (M_*/U_\alpha))$, with the
linear topology given by the kernels of the surjections $V_* \ctensor
M_* \to V_* \otimes (M_*/U_\alpha)$.  The completed tensor product is
complete Hausdorff by construction.  Given a second graded vector space
$W_*$, discrete and bounded below, there is a canonical isomorphism $(V_*
\otimes W_*) \ctensor M_* \cong V_* \ctensor (W_* \ctensor M_*)$.

\begin{dfn} \label{dfn_completecomodule}
Let $M_*$ be a complete Hausdorff graded $\F_p$-vector space.  We say that
$M_*$ is a \emph{complete $\A_*$-comodule} if there is a continuous graded
homomorphism $\nu \: M_* \to \A_* \ctensor M_*$ such that the diagrams
\[
\xymatrix{
M_* \ar[r]^-\nu \ar[dr]_{\cong} & \A_* \ctensor M_*
	\ar[d]^{\epsilon\ctensor1} \\
& \F_p \ctensor M_*
}
\]
and
\begin{equation} \label{equ_contcomod}
\xymatrix{
M_* \ar[r]^-\nu \ar[d]_\nu & \A_* \ctensor M_* \ar[dr]^{\psi\ctensor 1} \\
\A_* \ctensor M_* \ar[r]^-{1\ctensor \nu} &
  \A_* \ctensor (\A_* \ctensor M_*) \ar[r]^\cong &
  (\A_* \otimes \A_*) \ctensor M_*
}
\end{equation}
commute.  Here $\epsilon \: \A_* \to \F_p$ and $\psi \: \A_* \to \A_*
\otimes \A_*$ denote the counit and coproduct in the dual Steenrod
algebra, respectively.  Let $N_*$ be another complete $\A_*$-comodule
and let $f \: N_* \to M_*$ be a continuous graded homomorphism.  Then $f
\in \Hom_{\A_*}^c(N_*, M_*)$ if the diagram
$$
\xymatrix{
N_* \ar[r]^-\nu \ar[d]_f & \A_* \ctensor N_* \ar[d]^{1 \ctensor f} \\
M_* \ar[r]^-\nu & \A_* \ctensor M_*
}
$$
commutes.  Hence there is an equalizer diagram
\begin{equation} \label{equ_homaequalizer}
\xymatrix{
\Hom_{\A_*}^c(N_*, M_*) \to \Hom^c(N_*, M_*)
\ar@<0.6ex>[rr]^-{f\mapsto (1\ctensor f) \circ \nu}
\ar@<-0.6ex>[rr]_-{f\mapsto \nu\circ f} &&
\Hom^c(N_*, \A_* \ctensor M_*)\,.
}
\end{equation}
\end{dfn}

\begin{lemma}
Suppose given a sequence of graded $\F_p$-vector spaces, as in
\eqref{equ_cohomtower}, with each $A^n$ bounded below and of finite type.
Suppose also that $A^n$ is an $\A$-module and that $A^n \to A^{n-1}$
is $\A$-linear, for each finite $n$.  Then, with notation as above,
$A^{-\infty}$ is an $\A$-module, and the topological $\F_p$-vector space
$A_{-\infty}$ is a complete $\A_*$-comodule.
\end{lemma}

\begin{proof}
The category of $\A$-modules is closed under direct limits, so the first
claim of the lemma is immediate.  For each $n$ we get a commutative diagram
\[
\xymatrix{
\A \otimes A^n \ar@{->>}[r] \ar[d]_{\lambda^n} &
\A \otimes F^nA^{-\infty} \ar@{ >->}[r] \ar[d] &
\A \otimes A^{-\infty} \ar[d]^\lambda \\
A^n \ar@{->>}[r]  &
F^nA^{-\infty} \ar@{ >->}[r]  &
A^{-\infty} \rlap{\,,}
}
\]
where the vertical arrows are the $\A$-module action maps.
For every finite $n$, the dual of the $\A$-module action
map $\lambda^n \: \A \otimes A^n \to A^n$ defines an $\A_*$-comodule coaction map
$\nu_n \: A_n = \Hom(A^n, \F_p) \to \Hom(\A \otimes A^n, \F_p)
	\cong \Hom(\A, \F_p) \otimes \Hom(A^n, \F_p) = \A_* \otimes A_n$,
where the middle isomorphism uses that $\A$ and $A^n$ are bounded
below and of finite type over~$\F_p$.
Similarly, the dual of the diagram above gives a commutative
diagram
\[
\xymatrix{
\A_* \otimes A_n & \A_* \otimes A_{-\infty}/F_nA_{-\infty} \ar@{ >->}[l] &
	\Hom(\A \otimes A^{-\infty}, \F_p) \ar@{->>}[l] \\
A_n \ar[u]^{\nu_n} & A_{-\infty}/F_nA_{-\infty} \ar@{ >->}[l] \ar[u] &
	A_{-\infty} \ar@{->>}[l] \ar[u] \rlap{\,,}
}
\]
where we use the identification from~\eqref{equ_fna8dual}.  Passing
to limits over $n$, we get the diagram
\[
\xymatrix{
\displaystyle{\lim_n} \, (\A_* \otimes A_n) & \A_* \ctensor A_{-\infty}
\ar[l]_-\cong & \Hom(\A \otimes A^{-\infty}, \F_p) \ar[l]_-\cong \\
\displaystyle{\lim_n} \, A_n \ar[u]^{\lim_n \nu_n} & A_{-\infty}
\ar[l]_-\cong \ar[u]_\nu & A_{-\infty} \ar[l]_-{=}
	\ar[u]_{\Hom(\lambda, \F_p)} \rlap{\,.}
}
\]
The middle vertical coaction map $\nu$ is continuous, as it is realized as
the inverse limit of a homomorphism of towers.  It is clear that
the upper left hand horizontal map is injective, but we claim that it is
also surjective.

To see this, let $Z_n$ be the cokernel of $A_{-\infty}/F_nA_{-\infty}
\into A_n$.  We know from Lemma~\ref{lem_filtration} that $\lim_n
(A_{-\infty}/F_nA_{-\infty}) \cong \lim_n A_n$ and $\Rlim_n
(A_{-\infty}/F_nA_{-\infty}) = 0$, so $\lim_n Z_n = 0$.  This implies that
$\lim_n (\A_* \otimes Z_n) = 0$, since there are natural injective maps
$\A_* \otimes Z_n \into \Hom(\A, Z_n)$, and $\lim_n (\A_* \otimes Z_n)
\into \lim_n \Hom(\A, Z_n) \cong \Hom(\A, \lim_n Z_n) = 0$.  Now $\A_*
\otimes Z_n$ is the cokernel of $\A_* \otimes A_{-\infty}/F_nA_{-\infty}
\into \A_* \otimes A_n$, hence in the limit $\A_* \ctensor A_{-\infty}
\to \lim_n (\A_* \otimes A_n)$ is surjective.

The commutativity of the diagrams in Definition~\ref{dfn_completecomodule}
is immediate since they are obtained as the inverse limits of the
corresponding diagrams involving $A_n$ and $\nu_n$.  Thus $A_{-\infty}$
is a complete $\A_*$-comodule.
\end{proof}

\begin{cor} \label{cor_HcY}
Let $\{Y_n\}_n$ be a tower of spectra as in~\eqref{equ_tower}, each
bounded below and of finite type over $\F_p$, with homotopy limit $Y$.
Then the continuous cohomology $H_c^*(Y) = \colim_n H^*(Y_n)$ is an
$\A$-module, the continuous homology $H^c_*(Y) = \lim_n H_*(Y_n)$
is a complete $\A_*$-comodule, and there are natural isomorphisms
$\Hom(H_c^*(Y), \F_p) \cong H^c_*(Y)$ and $\Hom^c(H^c_*(Y), \F_p)
\cong H_c^*(Y)$, in the respective categories.
\qed
\end{cor}

\section{The algebraic Singer constructions} \label{sec_algsinger}
Classically, the algebraic Singer construction is an endofunctor on
the category of modules over the Steenrod algebra.  In \textsection
\ref{subsec_cohomalgsinger} we recall its definition, and a key property
proved by Adams, Gunawardena and Miller.   We then dualize the construction
in \textsection \ref{subsec_homalgsinger}.

Later, we will see how the algebraic Singer construction arises in its
cohomological (resp.~homological) form as the continuous cohomology
(resp.~continuous homology) of a certain tower of truncated Tate
spectra.  This tower of spectra induces a natural filtration on
the Singer construction.   We introduce this filtration in purely
algebraic terms in the present section, and will show in \textsection
\ref{subsec_tatessforsinger} that the algebraic and topological
definitions agree.

\subsection{The cohomological Singer construction} \label{subsec_cohomalgsinger}

\begin{dfn} \label{dfn_rplus}
Let $M$ be an $\A$-module.  The \emph{Singer construction}
$R_+(M)$ on $M$ is a graded $\A$-module given additively by the
formulas
$$
  \Sigma^{-1} R_+(M) = P(x,x^{-1}) \otimes M
$$
for $p=2$, and
$$
  \Sigma^{-1} R_+(M) = E(x) \otimes P(y, y^{-1}) \otimes M
$$
for $p$ odd.  Here $\deg(x)=1$, $\deg(y)=2$, and $\Sigma^{-1}$ denotes
desuspension by one degree.
The action of the Steenrod algebra is given, for $r \in \Z$ and
$a \in M$, by the formula
\begin{equation} \label{equ_singerops}
\Sq^s(x^r \otimes a) = \sum_j
	\binom{r-j}{s-2j} \, x^{r+s-j} \otimes \Sq^j(a)
\end{equation}
for $p=2$, and the formulas
\begin{align*} \label{equ_oddsingeroperations}
\SP^s(y^r \otimes a)
	&= \sum_j \binom{r-(p-1)j}{s-pj}
	\, y^{r+(p-1)(s-j)} \otimes \SP^j(a) \\
    &\qquad + \sum_j \binom{r-(p-1)j-1}{s-pj-1}
	\, x y^{r+(p-1)(s-j)-1} \otimes \beta\SP^j(a) \\
\SP^s(x y^{r-1} \otimes a) &=
	\sum_j \binom{r-(p-1)j-1}{s-pj}
	\, x y^{r+(p-1)(s-j)-1} \otimes \SP^j(a)
\end{align*}
and
\begin{align*}
\beta(y^r \otimes a) &= 0 \\
\beta(x y^{r-1} \otimes a) &= y^r \otimes a
\end{align*}
for $p$ odd.
\end{dfn}

This is the form of the Singer construction that is related to the
cyclic group $C_p$.  The cohomology of the classifying space of
this group is $H^*(BC_p) \cong E(x) \otimes P(y)$ for $p$ odd, with
$\deg(x)=1$, $\deg(y)=2$ and $\beta(x) = y$, as above.  The natural
$\A$-module structure on $H^*(BC_p)$ extends to the localization
$H^*(BC_p)[y^{-1}] = E(x) \otimes P(y, y^{-1})$, and letting $M = \F_p$
we get that $\Sigma^{-1} R_+(\F_p)$ is isomorphic to $H^*(BC_p)[y^{-1}]$
as an $\A$-module.  The case $p=2$ is similar.

When $p$ odd there is a second form of the Singer construction, related
to the symmetric group $\Sigma_p$.  Following \cite{LS1982}*{p.~272}
we identify $H^*(B\Sigma_p)$ with the subalgebra $E(u) \otimes P(v)$
of $H^*(BC_p)$ generated by $u = -x y^{p-2}$ and $v = - y^{p-1}$, with
$\deg(u) = 2p-3$ and $\deg(v) = 2p-2$.  The smaller form of the Singer
construction then corresponds to the direct summand $E(u) \otimes P(v,
v^{-1}) \otimes M$ of index~$(p-1)$ in $E(x) \otimes P(y, y^{-1}) \otimes
M$.  Explicit formulas for action of the Steenrod operations on the
smaller form of the Singer construction are given in \cite{S1981}*{(3.2)},
\cite{LS1982}*{\textsection 2} and \cite{BMMS}*{p.~47}.

In our work, we are only concerned with the version of the Singer
construction related to the group $C_p$.  The exact form of the formulas
in Definition~\ref{dfn_rplus} is justified by Theorem~\ref{thm_TopSinger}
below.

\subsubsection{The cohomological $\epsilon$-map} \label{subsec_algepsilon}
An important property of $R_+(M)$ is that there exists a natural
homomorphism $\epsilon \: R_+(M) \to M$ of $\A$-modules.  In Singer's
original definition for $p=2$, the map is given by the formula
\begin{equation} \label{equ_epsilon_even_formula}
\epsilon(\Sigma x^{r-1} \otimes a) =
\begin{cases}
\Sq^r(a) & \text{for $r\ge0$,} \\
0 & \text{for $r<0$.}
\end{cases}
\end{equation}
For $p$ odd, the $\A$-submodule spanned by elements of the form
$\Sigma y^{(p-1)r} \otimes a$ or
$\Sigma x y^{(p-1)r-1} \otimes a$ is a direct summand in $R_+(M)$.
The homomorphism $\epsilon$ is given by first projecting onto this direct
summand and then composing with the map
\begin{equation} \label{equ_epsilon_odd_formula}
\begin{split}
\Sigma y^{(p-1)r} \otimes a &\mapsto -(-1)^r \beta \SP^r(a) \\
\Sigma x y^{(p-1)r-1} \otimes a &\mapsto (-1)^r \SP^r(a)
\end{split}
\end{equation}
for $r\ge0$, still mapping to $0$ for $r<0$.  See \cite{BMMS}*{p.~50}.
It is clear that $\epsilon$ is surjective.

We recall the key property of $\epsilon$.  Adams, Gunawardena and Miller
\cite{AGM1985} make the following definition.

\begin{dfn}
An $\A$-module homomorphism $L \to M$ is a \emph{$\Tor$-equivalence}
if the induced map
\begin{equation}
	\Tor^{\A}_{*,*}(\F_p, L) \to \Tor^{\A}_{*,*}(\F_p, M)
\end{equation}
is an isomorphism.
\end{dfn}

The relevance of this condition is:

\begin{prop}[\cite{AGM1985}*{1.2}] \label{prop_extiso}
If $L \to M$ is a $\Tor$-equivalence, then for any $\A$-module~$N$
that is bounded below and of finite type the induced map
\begin{equation}
    \Ext_\A^{*,*}(M, N)\to \Ext_{\A}^{*,*}(L, N)
\end{equation}
is an isomorphism.
\end{prop}

Here is their key result, proved in \cite{AGM1985}*{1.3}.

\begin{thm}[Gunawardena, Miller] \label{thm_extiso}
The Singer homomorphism $\epsilon \: R_+(M) \to M$ is a $\Tor$-equivalence.
\end{thm}

We will later encounter instances of $\A$-module homomorphisms $R_+(M)
\to M$ induced by maps of spectra.  It is often possible to determine
those homomorphisms by the following corollary.

\begin{cor} \label{cor_uniquehom}
Let $M$, $N$ be $\A$-modules such that $N$ is bounded below and of
finite type.  Then
$$
\epsilon^* \: \Hom_{\A}(M, N) \to \Hom_{\A}(R_+(M), N)
$$
is an isomorphism, so any $\A$-linear homomorphism $f \: R_+(M) \to N$
factors as $g \circ \epsilon$ for a unique $\A$-linear homomorphism $g
\:M \to N$:
$$
\xymatrix{
R_+(M) \ar[rr]^-\epsilon \ar[dr]_f && M \ar[dl]^g \\
& N
}
$$
\end{cor}

\begin{proof}
This is clear from Theorem~\ref{thm_extiso} and
Proposition~\ref{prop_extiso}.
\end{proof}

\begin{remark}
A special case of this occurs when $M = N$ is a cyclic $\A$-module.
Then$$
\F_p \cong \Hom_{\A}(M, M) \cong \Hom_{\A}(R_+(M), M) \,,
$$
so any $\A$-linear homomorphism $R_+(M) \to M$ is equal to a scalar
multiple of $\epsilon$.
\end{remark}

\subsection{The homological Singer construction} \label{subsec_homalgsinger}
Before we define the homological version of the Singer construction on
an $\A_*$-comodule $M_*$, we need to discuss a natural filtration on
the cohomological Singer construction.  For a bounded below
$\A$-module $M$ of finite type over $\F_p$, let
$$
F^nR_+(M) = \F_2\{ \Sigma x^r \otimes a \mid
	r \in \Z, \ \deg(a) = q, \ 1+r-q \ge n\}
$$
for $p=2$, and
$$
F^nR_+(M) = \F_p\{\Sigma x^i y^r \otimes a \mid
	i \in \{0,1\}, \ r \in \Z, \ \deg(a) = q, \ 1+i+2r-(p-1)q \ge n\}
$$
for $p$ odd.  In each case $a$ runs through an $\F_p$-basis for $M$.
Then
\begin{equation} \label{equ_rplusfiltration}
   \cdots \subset F^nR_+(M)\subset F^{n-1}R_+(M)\subset \cdots \subset R_+(M)
\end{equation}
is an exhaustive filtration of $R_+(M)$, which is clearly Hausdorff.
Because $M$ is bounded below and of finite type, each $F^nR_+(M)$ is
bounded below and of finite type, so $\Rlim_n F^nR_+(M)$ is trivial.
Hence the filtration is complete.

For reasons made clear in Corollary~\ref{cor_comparefiltrations}, we will
refer to this filtration as the \emph{Tate filtration}.  When $M$ is the
cohomology of a bounded below spectrum of finite type over~$\F_p$, we will
see how \eqref{equ_rplusfiltration} is induced from topology.  In this
case, it will be immediate that the filtration is one of $\A$-modules.
For a general $\A$-module $M$, this can be checked directly using the
explicit formulas in Definition~\ref{dfn_rplus}.

We are now in the situation discussed in the previous section, with
$A^n = F^nR_+(M)$ and $A^{-\infty} = R_+(M)$.  Letting $F^nR_+(M)_* =
\Hom(F^nR_+(M), \F_p) = A_n$ we get an inverse system
\begin{equation} \label{equ_singhomfilt}
  \cdots \to F^{n-1}R_+(M)_* \to F^nR_+(M)_* \to \cdots
\end{equation}
as in~\eqref{equ_homtower}, dual to the direct system
\eqref{equ_rplusfiltration}.  We are interested in the inverse limit
$A_{-\infty} = \lim_n A_n$, with the linear topology given by this
tower of surjections.  Recall Definition~\ref{dfn_completecomodule}
of a complete $\A_*$-comodule.

\begin{dfn} \label{dfn_homologysinger}
Let $M_*$ be a bounded below $\A_*$-comodule of finite type.
Its dual $M = \Hom(M_*, \F_p)$ is a bounded below
$\A$-module of finite type, and $M_* \cong \Hom(M, \F_p)$.  We define
the \emph{homological Singer construction} on $M_*$ to be the complete
$\A_*$-comodule given by
$$
R_+(M_*) = \Hom(R_+(M), \F_p) \,.
$$
It is isomorphic to the inverse limit $\lim_n F^nR_+(M)_*$.
\end{dfn}

A more explicit description can be given.  For $p=2$ the $\F_2$-linear
dual of
\[
\tH_{-*}(C_2; \F_2) \cong \Sigma H^*(C_2; \F_2)[x^{-1}] = \Sigma P(x, x^{-1})
\]
is isomorphic to the ring of Laurent polynomials $\tH^{-*}(C_2; \F_2)
= P(u, u^{-1})$, where $\deg(u)=-1$ and $u^{-r}$ is dual to $\Sigma
x^{r-1}$.  For $p$ odd, the $\F_p$-linear dual of
\[
\tH_{-*}(C_p; \F_p) \cong \Sigma H^*(C_p; \F_p)[y^{-1}]
	= \Sigma E(x) \otimes P(y, y^{-1})
\]
is isomorphic to $\tH^{-*}(C_p; \F_p) = E(u) \otimes P(t, t^{-1})$,
where $\deg(u)=-1$, $\deg(t)=-2$ and $u^{1-i} t^{-r}$ is dual to $\Sigma x^i
y^{r-1}$.  These notations are compatible with those from \cite{BM1994}.
We get the following identifications:
\[
F^nR_+(M)_* \cong \F_2\{ u^r \otimes \alpha \mid
	r \in \Z, \ \deg(\alpha) = q, \ r+q \le -n\}
\]
for $p=2$, and
$$
F^nR_+(M)_* \cong \F_p\{ u^i t^r\otimes \alpha \mid
	i \in \{0,1\}, \ r \in \Z, \ \deg(\alpha) = q, \ i+2r+(p-1)q \le -n\}
$$
for $p$ odd.
In each case $\alpha$ ranges over an $\F_p$-basis for $M_*$.
The maps of~\eqref{equ_singhomfilt} are given by the obvious projections.
Thus, $R_+(M_*)$ is isomorphic to the graded vector space of formal series
\[
\sum_{r=-\infty}^\infty u^r \otimes \alpha_r
\]
for $p=2$, and
\[
\sum_{r=-\infty}^\infty t^r \otimes \alpha_{0,r}
+ \sum_{r=-\infty}^\infty u t^r \otimes \alpha_{1,r}
\]
for $p$ odd.  In each of these sums $r$ is bounded below, but not above,
since $M_*$ is bounded below.

Using the linear topology on $R_+(M_*)$ given by the kernel filtration
coming from~\eqref{equ_singhomfilt}, we may reformulate this as follows:
Let
$$
\Lambda = \tH^{-*}(C_p; \F_p) = \begin{cases}
P(u, u^{-1}) & \text{for $p=2$,} \\
E(u) \otimes P(t,t^{-1}) & \text{for $p$ odd.}
\end{cases}
$$
Consider $\Lambda\otimes M_*\subset R_+(M_*)$.  For every $n$ the
composition $\Lambda\otimes M_* \subset R_+(M_*) \to F^nR_+(M)_*$ is
surjective, so the completed tensor product $\Lambda\ctensor M_*$ (for
the linear topology on $\Lambda$ derived from the grading) is canonically
isomorphic to $R_+(M_*)$.

\subsubsection{The homological $\epsilon_*$-map}
Let
$$
\epsilon_* \: M_* \to R_+(M_*)
$$
be the dual of $\epsilon\: R_+(M)\to M$.  Then $\epsilon_*$ is a
continuous homomorphism of complete $\A_*$-comodules.  Continuity is
trivially satisfied since the source of $\epsilon_*$ has the discrete
topology.

Dualizing~\eqref{equ_epsilon_even_formula} and
\eqref{equ_epsilon_odd_formula}, we see that $\epsilon_*$ is given by
the formulas
\begin{equation}
	\label{equ_dualSingerEvaluation_even}
\epsilon_*(\alpha) = \sum_{r=0}^\infty u^{-r} \otimes \Sq_*^r(\alpha)
\end{equation}
for $p=2$, and
\begin{equation}
	\label{equ_dualSingerEvaluation_odd}
\epsilon_*(\alpha) = \sum_{r=0}^\infty
	(-1)^r t^{-(p-1)r} \otimes \SP^r_*(\alpha)
- \sum_{r=0}^\infty
	(-1)^r u t^{-(p-1)r-1} \otimes (\beta\SP^r)_*(\alpha)
\end{equation}
for $p$ odd.
This expression may be compared with \cite{AGM1985}*{(3.6)}.
It is clear that $\epsilon_*$ is injective.

\begin{lemma} \label{lem_homiso}
Let $M$ and $N$ be bounded below $\A$-modules of finite type, and let
$M_*$ and $N_*$ be the dual $\A_*$-comodules.  Then
$$
\epsilon_*\: \Hom_{\A_*}(N_*, M_*) \to \Hom^c_{\A_*}(N_*, R_+(M_*))
$$
is an isomorphism, so any continuous $\A_*$-comodule homomorphism
$f_* \: N_* \to R_+(M_*)$ factors as $f_* = \epsilon_* \circ g_*$
for a unique $\A_*$-comodule homomorphism $g_* \: N_* \to M_*$.
\end{lemma}

\begin{proof}
Notice that $\Hom_{\A_*}(N_*, M_*) = \Hom^c_{\A_*}(N_*, M_*)$ and $\A_*
\otimes N_* = \A_* \ctensor N_*$, since $M_*$ and $N_*$ are discrete.
Applying $\Hom(-, \F_p)$ to a commutative square
\[
\xymatrix{
\A \otimes R_+(M) \ar[r]^-\lambda \ar[d]_{1 \otimes f} & R_+(M) \ar[d]^f \\
\A \otimes N \ar[r]^-\lambda & N
}
\]
we get a commutative square
\[
\xymatrix{
\A_* \ctensor R_+(M_*) & R_+(M_*) \ar[l]_-\nu \\
\A_* \otimes N_* \ar[u]^{1\ctensor f_*} & N_* \ar[l]_-\nu \ar[u]_{f_*}
}
\]
of continuous homomorphisms, where $R_+(M_*)$ and $\A_* \ctensor
R_+(M_*)$ have the limit topologies, while $N_*$ and $\A_* \otimes N_*$
are discrete.  Applying $\Hom^c(-, \F_p)$ to the latter square we recover
the first, by Lemma~\ref{lem_contdual}.  Hence the right hand vertical
map in the commutative square
\[
\xymatrix{
\Hom_{\A}(M, N) \ar[r]^-{\epsilon^*}_-\cong \ar[d]_\cong &
\Hom_{\A}(R_+(M), N) \ar[d]^\cong \\
\Hom_{\A_*}(N_*, M_*) \ar[r]^-{\epsilon_*}  &
\Hom_{\A_*}^c(N_*, R_+(M_*))
}
\]
is an isomorphism.  It is easy to see that the left hand vertical map
is an isomorphism, and the upper horizontal map is an isomorphism by
Corollary~\ref{cor_uniquehom}.
\end{proof}

\subsubsection{Various remarks on the homological Singer construction}
The following remarks are not necessary for our immediate applications,
but we include them to shed some light on the coaction $\nu \: R_+(M_*)
\to \A_* \ctensor R_+(M_*)$ and the dual Singer map $\epsilon_* \: M_*
\to R_+(M_*)$, and their relations to the completions introduced so far.

Dualizing~\eqref{equ_singerops}, we get that the dual Steenrod operations
on classes $u^r \otimes \alpha$ in $\Lambda \otimes M_* \subset R_+(M_*)$
are given by
\begin{equation} \label{equ_dualSinger}
\Sq_*^s(u^r \otimes \alpha) = \sum_j \binom{-r-s-1}{s-2j}
	u^{r+s-j} \otimes \Sq_*^j(\alpha)
\end{equation}
for $p=2$, and similarly for $p$ odd.
This sum is finite, since $M_*$ is assumed to be bounded below, so we
have the following commutative diagram:
\begin{equation} \label{diag_subtensor}
\xymatrix{
R_+(M_*) \ar[r]^-\nu & \A_* \ctensor R_+(M_*) \\
\Lambda \otimes M_* \ar@{ >->}[u] \ar[r] &
\A_* \ctensor (\Lambda \otimes M_*) \ar@{ >->}[u]
}
\end{equation}
Two remarks are in order.  First, $\Lambda \otimes M_*$ is not complete
with respect to the subspace topology from $R_+(M_*)$.  Hence $\Lambda
\otimes M_*$ is not a complete $\A_*$-comodule in the sense explained
above.  Second, there are elements $u^r \otimes \alpha$ in $\Lambda
\otimes M_*$ with the property that $\Sq^s_*(u^r \otimes \alpha)$ is
nonzero for infinitely many $s$, and similarly for $p$ odd.  For example,
$\Sq^s_*(u^{-1} \otimes \alpha)$ contains the term
\[
\binom{-s}{s} u^{s-1} \otimes \alpha
= \binom{2s-1}{s} u^{s-1} \otimes \alpha
\]
for $j=0$, according to~\eqref{equ_dualSinger}.  This equals $u^{s-1}
\otimes \alpha$ whenever $s = 2^e$ is a power of~$2$, so $\nu(u^{-1}
\otimes \alpha)$ is an infinite sum.  Hence $\Lambda \otimes M_*$ is
not an algebraic $\A_*$-comodule, either.

We will now identify the image of the homological version of the Singer
map
$$
  \epsilon_* \: M_* \to R_+(M_*)
$$
with the maximal algebraic $\A_*$-comodule contained in $R_+(M_*)$.

\begin{dfn}
Given a complete $\A_*$-comodule $N_*$, let $N_*^{\alg} \subseteq N_*$
be given by the pullback
\[
\xymatrix{
N_*^{\alg} \ar[rr]^-{\nu|N_*^{\alg}} \ar@{ >->}[d]
	&& \A_* \otimes N_* \ar@{ >->}[d] \\
N_* \ar[rr]^-{\nu} && \A_* \ctensor N_*
}
\]
in graded $\F_p$-vector spaces.
In other words, $N_*^{\alg}$ consists of the $\alpha \in N_*$ whose
coaction $\nu(\alpha) = \sum_I \Sq^I_* \otimes \alpha_I$ (in the notation
for $p=2$) is a finite sum, rather than a formal infinite sum.  Here $I$
runs over the admissible sequences, so that $\{\Sq^I_*\}_I$ is a basis
for $\A_*$, and $\alpha_I = \Sq^I_*(\alpha)$.
\end{dfn}

\begin{lemma}
The restricted coaction map $\nu|N_*^{\alg}$ factors (uniquely) through
the inclusion $\A_* \otimes N_*^{\alg} \subseteq \A_* \otimes N_*$,
hence defines a map
\[
\nu^{\alg} \: N_*^{\alg} \to \A_* \otimes N_*^{\alg}
\]
that makes $N_*^{\alg}$ an $\A_*$-comodule in the algebraic sense.
\end{lemma}

\begin{proof}
The composite
\[
\xymatrix{
N_*^{\alg} \ar[r]^-{\nu|N_*^{\alg}} &
\A_* \otimes N_* \ar[r]^-{1 \otimes \nu} &
\A_* \otimes (\A_* \ctensor N_*)
}
\]
factors as
\[
\xymatrix{
N_*^{\alg} \ar[r]^-{\nu|N_*^{\alg}} &
\A_* \otimes N_* \ar[r]^-{\psi \otimes 1} &
\A_* \otimes \A_* \otimes N_*
\subseteq
\A_* \otimes (\A_* \ctensor N_*)
}
\]
by coassociativity~\eqref{equ_contcomod} of the complete coaction.  Hence
$\nu|N_*^{\alg}$ factors through the pullback $\A_* \otimes N_*^{\alg}$ in
\[
\xymatrix{
\A_* \otimes N_*^{\alg} \ar[rr]^-{1 \otimes \nu|N_*^{\alg}} \ar@{ >->}[d] &&
	\A_* \otimes \A_* \otimes N_* \ar@{ >->}[d] \\
\A_* \otimes N_* \ar[rr]^-{1 \otimes \nu} && \A_* \otimes (\A_* \ctensor N_*)
	\rlap{\,.}
}
\]
Algebraic counitality and coassociativity of the lifted map
$\nu^{\alg}$ follow from the corresponding complete properties of $\nu$. 
\end{proof}

The following identification stems from a conversation with
M.~B{\"o}kstedt.

\begin{prop}
The image of the injective homomorphism $\epsilon_* \: M_* \to R_+(M_*)$
equals the maximal algebraic sub $\A_*$-comodule $R_+(M_*)^{\alg} \subset
R_+(M_*)$.
\end{prop}

\begin{proof}
Let $L_*$ be any algebraic $\A_*$-comodule.  Given any $\alpha
\in L_*$, with coaction $\nu(\alpha) = \sum_I \Sq^I_* \otimes \alpha_I$,
let $\langle \alpha \rangle \subseteq L_*$ be the graded vector subspace
spanned by the $\alpha_I = \Sq^I_*(\alpha)$.  Here we are using the
notation appropriate for $p=2$; the case $p$ odd is completely similar.
Since $\nu(\alpha)$ is a finite sum, $\langle \alpha \rangle$ is a
finite dimensional subspace.  Furthermore, it is a sub $\A_*$-comodule,
since $\nu(\alpha_I) = \sum_J \Sq^J_* \otimes \Sq^J_*(\alpha_I)$ and
$\Sq^J_*(\alpha_I) = (\Sq^I \Sq^J)_*(\alpha)$ is a finite sum of terms
$\Sq^K_*(\alpha) = \alpha_K$.

Now consider the case $L_* = R_+(M_*)^{\alg}$.  It is clear that
$\epsilon_*(M_*) \subseteq R_+(M_*)^{\alg}$, since $M_*$ is an algebraic
$\A_*$-comodule and $\epsilon_*$ respects the coaction.  Let $\alpha
\in R_+(M_*)^{\alg}$ be any element, and consider the linear span
\[
N_* = \epsilon_*(M_*) + \langle \alpha \rangle \subseteq R_+(M_*)^{\alg} \,.
\]
It is bounded below and of finite type, so by Lemma~\ref{lem_homiso}
there is a unique lift $g_*$
\[
\xymatrix{
M_* \ar[rr]^-{\epsilon_*} & & R_+(M_*) \\
& N_* \ar[ul]^{g_*} \ar@{ >->}[ur]_{f_*}
}
\]
of the inclusion $f_* \: N_* \to R_+(M_*)$.
Hence $N_* \subseteq \epsilon_*(M_*)$, so in fact $\alpha \in
\epsilon_*(M_*)$.
\end{proof}

\section{The Tate construction} \label{sec_tate}
We recall the Tate construction of Greenlees, and its relation with
homotopy orbit and homotopy fixed point spectra.  We then show how it
can be expressed as the homotopy inverse limit of bounded below spectra,
in two equivalent ways.  This lets us make sense of the continuous
(co-)homology groups of the Tate construction.

We then describe the homological Tate spectral sequences.  There are two
types, one converging to the continuous homology of the Tate construction
and one converging to the continuous cohomology.  The terms of these spectral
sequences will be linearly dual to each other, but, as already noted in
\textsection \ref{sec_duality}, their target groups will only be dual
in a topologized sense.
The main properties of these spectral sequences are summarized
in Propositions~\ref{prop_cohomological}, \ref{prop_homological}
and~\ref{prop_homological2}.

\subsection{Equivariant spectra and various fixed point constructions}
We review some notions from stable equivariant homotopy theory, in the
framework of Lewis--May spectra \cite{LMS}.
Let $G$ be a compact Lie group, quite possibly finite, and let $\U$
be a complete $G$-universe.  We fix an identification $\U^G = \R^\infty$,
and write $i \: \R^\infty \to \U$ for the inclusion.

Let $G\S\U$ be the category of genuine $G$-spectra, and let
$G\S\R^\infty$ be the category of naive $G$-spectra.  Similarly, let
$\S\R^\infty$ be the category of (non-equivariant) spectra.
The restriction of universe functor $i^* \: G\S\U \to G\S\R^\infty$ has
a left adjoint, the extension of universe functor $i_* \: G\S\R^\infty
\to G\S\U$, see \cite{LMS}*{\textsection II.1}.

The functor $\S\R^\infty \to G\S\R^\infty$, giving a spectrum the
trivial $G$-action, has a left adjoint taking a naive $G$-spectrum $Y$
to the orbit spectrum $Y/G$, as well as a right adjoint taking $Y$ to the
fixed point spectrum $Y^G$.  For a genuine $G$-spectrum~$X$, the orbit
spectrum $X/G = (i^* X)/G$ and fixed point spectrum $X^G = (i^* X)^G$
are defined by first restricting to the underlying naive $G$-spectra.

Let $EG$ be a free, contractible $G$-CW complex.  Let $c \: EG_+ \to
S^0$ be the collapse map that sends $EG$ to the non-base point of $S^0$,
and let $\tEG$ be its mapping cone, so that we have a homotopy cofiber
sequence
\begin{equation} \label{equ_fundamentalsequence}
  EG_+ \overset{c}\to S^0 \to \tEG
\end{equation}
of based $G$-CW complexes.  The $n$-skeleton $\tEG{}^{(n)}$ of $\tEG$ is
then the mapping cone of the restricted collapse map $EG^{(n-1)}_+
\to S^0$, for each $n\ge0$.  We may and will assume that each skeleton
$EG^{(n-1)}$ is a finite $G$-CW complex.

\begin{dfn}
For each naive $G$-spectrum $Y$ let
$Y_{hG} = (EG_+ \wedge Y)/G$
be the \emph{homotopy orbit spectrum}, and let
$Y^{hG} = \map(EG_+, Y)^G$
be the \emph{homotopy fixed point} spectrum.
For each genuine $G$-spectrum $X$ let
$$
X_{hG} = (EG_+ \wedge i^* X)/G = (i^* X)_{hG}
$$
and
$$
X^{hG} = \map(EG_+, X)^G = (i^* X)^{hG}
$$
be defined by first restricting to the $G$-trivial universe.
Furthermore, let
$$
X^{tG} = [\tEG \wedge \map(EG_+, X)]^G
$$
be the \emph{Tate construction} on $X$.  This is the spectrum denoted
$\widehat{\mathbb{H}}(G, X)$ by B{\"o}kstedt and Madsen \cite{BM1994}
and $t_G(X)^G$ by Greenlees and May \cite{GM1995}.
\end{dfn}

The Segal conjecture is concerned with the map $\Gamma \: X^G
\to X^{hG}$ induced by $\map(c, 1) \: X \cong \map(S^0, X) \to
\map(EG_+, X)$ by passing to fixed points.  By smashing the cofiber
sequence~\eqref{equ_fundamentalsequence} with $\map(c, 1)$ and passing
to $G$-fixed points, we can embed this map in the following diagram,
consisting of two horizontal cofiber sequences:
$$
\xymatrix{
[EG_+ \wedge X]^G \ar[r] \ar[d]_\simeq &
X^G \ar[r] \ar[d]^{\Gamma} &
[\tEG \wedge X]^G \ar[d]^{\hat\Gamma} \\
[EG_+ \wedge \map(EG_+, X)]^G \ar[r] &
\map(EG_+, X)^G \ar[r] &
[\tEG \wedge \map(EG_+, X)]^G
}
$$
The adjunction counit $\epsilon \: i_* i^* X \to X$ and the map
$\map(c, 1)$ are both $G$-maps and non-equivariant equivalences.  By the
$G$-Whitehead theorem, both maps
$$
1\wedge\epsilon \: i_*(EG_+ \wedge i^*X) = EG_+ \wedge i_* i^* X
	\to EG_+ \wedge X
$$
and
$$
1\wedge\map(c,1) \: EG_+ \wedge X \to EG_+ \wedge \map(EG_+, X)
$$
are genuine $G$-equivalences.  Hence we have the equivalence indicated on
the left.  Furthermore, there is an Adams transfer equivalence
\begin{equation} \label{equ_adamstransfer}
\tilde\tau \: (\Sigma^{\ad G} EG_+ \wedge i^* X)/G
\overset{\simeq}\longto [i_*(EG_+ \wedge i^* X)]^G \,,
\end{equation}
where $\ad G$ denotes the adjoint representation of $G$.  See
\cite{LMS}*{\textsection II.2} and \cite{GM1995}*{Part~I} for further
details.

In the cases of interest to us, when $G$ is discrete or abelian,
the adjoint representation is trivial so that $\Sigma^{\ad
G} = \Sigma^{\dim G}$.  Hence we may rewrite the diagram above as the
following \emph{norm--restriction} diagram
\begin{equation} \label{equ_fund1}
\xymatrix{
\Sigma^{\dim G} X_{hG} \ar[rr]^-N \ar[d]_{=} &&
X^G \ar[rr]^-R \ar[d]^{\Gamma} &&
[\tEG \wedge X]^G \ar[d]^{\hat\Gamma} \\
\Sigma^{\dim G} X_{hG} \ar[rr]^-{N^h} &&
X^{hG} \ar[rr]^-{R^h} &&
X^{tG}
}
\end{equation}
for any genuine $G$-spectrum $X$.
We note that the adjunction counit $\epsilon \: i_* i^* X \to X$ induces
equivalences $(i_* i^* X)_{hG} \simeq X_{hG}$ and $(i_* i^* X)^{hG} \simeq
X^{hG}$, hence $(i_* i^* X)^{tG} \simeq X^{tG}$, so the Tate construction
on $X$ only depends on the naive $G$-spectrum underlying $X$.

The spectra in the lower row have been studied by means of spectral
sequences converging to their homotopy groups, e.g.~in \cite{BM1994},
\cite{HM1997}, \cite{R1999}, \cite{AR2002} and \cite{HM2003}.  These
spectral sequences arise in the case of the homotopy orbit and fixed
point spectra by choosing a filtration of $EG$, and by a filtration of
$\tEG$ introduced by Greenlees \cite{G1987} in the case of the Tate
spectrum $X^{tG}$.  We shall instead be concerned with the spectral
sequences that arise by applying homology in place of homotopy.

\subsection{Tate cohomology and the Greenlees filtration of $\tEG$}
We recall the definition of the Tate cohomology groups from
\cite{CE}*{XII.3}, and the associated Tate homology groups.  Let $G$
be a finite group, let $\F_pG = \F_p[G]$ be its group algebra, and
let $(P_*, d_*)$ be a \emph{complete resolution} of the trivial $\F_pG$-module
$\F_p$ by free $\F_pG$-modules.  This is a commutative diagram
$$
\xymatrix{
\cdots \ar[r] & P_1 \ar[r]^{d_1} & P_0 \ar[r]^{d_0} \ar@{->>}[d]
  & P_{-1} \ar[r]^{d_{-1}} & P_{-2} \ar[r] & \cdots \\
&& \F_p \ar@{ >->}[ur]
}
$$
of $\F_pG$-modules, where the $P_n$'s are free and the horizontal sequence
is exact.  The image of $d_0$ is identified with $\F_p$, as indicated.

\begin{dfn}
Given an $\F_pG$-module $M$ the \emph{Tate cohomology} and \emph{Tate
homology groups} are defined by
$$
\tH^n(G; M) = H^n(\Hom_{\F_pG}(P_*, M))
$$
and
$$
\tH_n(G; M) = H_n(P_* \otimes_{\F_pG} M) \,,
$$
respectively, where $(P_*, d_*)$ is a complete $\F_pG$-resolution.
(To form the balanced tensor product, we turn $P_*$ into a complex of
right $\F_pG$-modules by means of the group inverse.)  These groups
are independent of the chosen complete $\F_pG$-resolution, and there
are isomorphisms
$$
\tH^n(G; M) \cong \tH_{-n-1}(G; M)
$$
and
$$
\Hom(\tH_n(G; M), \F_p) \cong \tH^n(G; \Hom(M, \F_p))
$$
for all integers $n$.  Note that we do not follow the shifted grading
convention for Tate homology given in \cite{GM1995}*{11.2}.
\end{dfn}

The topological analogue of a complete resolution is a bi-infinite
filtration of $\tEG$, in the category of $G$-spectra, which was introduced
by Greenlees \cite{G1987}.  We recall the details of the construction.
For brevity we shall not distinguish notationally between a
based $G$-CW complex and its suspension $G$-CW spectrum.  For integers
$n\ge0$ we let $\tF{n} = \tEG{}^{(n)}$ be (the suspension spectrum of) the
$n$-skeleton of $\tEG$, while $\tF{-n} = D(\tF{n}) = \map(\tEG{}^{(n)}, S)$
is its functional dual.  These definitions agree for $n=0$, as $\tF{0}
= S$ is the sphere spectrum.  Splicing the skeleton filtration of $\tEG$
with its functional dual, we get the finite terms in the following diagram
\begin{equation} \label{equ_greenleesfiltration}
D(\tEG) \to \dots \to \tF{-1} \to \tF{0} = S
	\to \tF{1} \to \tF{2} \to \dots \to \tEG \,,
\end{equation}
which we call the \emph{Greenlees filtration}.
Both $\tEG \simeq \hocolim_n \tF{n}$ and $D(\tEG) \simeq \holim_n \tF{n}$
are non-equivariantly contractible.

Applying homology to this filtration gives a spectral sequence with
$E^1_{s,t}=H_{s+t}(\tF{s} / \tF{s-1})$ that converges to $H_*(\tEG,
D(\tEG)) = 0$.  It is concentrated on the horizontal axis, since
$\tF{n}/\tF{n-1}$ is a finite wedge sum of $G$-free $n$-sphere spectra
$G_+ \wedge S^n$ for each integer $n$.  Hence the spectral sequence
collapses at the $E^2$-term, and we get a long exact sequence
\begin{equation} \label{equ_geometricresolution}
\xymatrix{
\dots \ar[r] & H_2(\tF{2}/\tF{1}) \ar[r]^-{d^1_{2,0}} & H_1(\tF{1}/\tF{0})
\ar[r]^-{d^1_{1,0}} \ar@{->>}[d] & H_0(\tF{0}/\tF{-1}) \ar[r] & \dots \\
&& H_0(S) \ar@{ >->}[ru]
}
\end{equation}
of finitely generated free $\F_pG$-modules.  Letting
$$
P_n = H_{n+1}(\tF{n+1} / \tF{n})
$$
and $d_n = d^1_{n+1,0}$ for all integers $n$ yields a complete resolution
$(P_*, d_*)$ of $\F_p = H_0(S)$.

\subsection{Continuous homology of the Tate construction}
\label{sec_tatefiltration}
Let $G$ be a finite group and let $X$ be a genuine $G$-spectrum.  By means
of the Greenlees filtration, we may filter the Tate construction $X^{tG}$
by a tower of spectra.

\begin{dfn}
For each integer $n$ let $\tEG/\tF{n-1}$ be the homotopy cofiber of the
map $\tF{n-1} \to \tEG$, and define
\begin{align*}
X^{tG}[-\infty,n{-}1] &= [\tF{n-1} \wedge \map(EG_+, X)]^G \\
X^{tG}[n] = X^{tG}[n, \infty] &= [\tEG/\tF{n-1} \wedge \map(EG_+, X)]^G \,.
\end{align*}
\end{dfn}

Smashing the cofiber sequence $\tF{n-1} \to \tEG \to \tEG/\tF{n-1}$
with $\map(EG_+, X)$ and taking $G$-fixed points, we get a cofiber
sequence
$$
X^{tG}[-\infty,n{-}1] \to X^{tG} \to X^{tG}[n, \infty]
$$
for each integer $n$.  The maps $\tF{n-1} \to \tF{n}$ in the Greenlees
filtration~\eqref{equ_greenleesfiltration} induce maps between these
cofiber sequences, which combine to the ``finite $n$ parts'' of the
following horizontal tower of vertical cofiber sequences:
\begin{equation} \label{equ_tatetowers}
\xymatrix{
{*} \ar[r] \ar[d] &
\dots \ar[r] &
X^{tG}[-\infty,n{-}1] \ar[r] \ar[d] &
X^{tG}[-\infty,n] \ar[r] \ar[d] &
\dots \ar[r] &
X^{tG} \ar[d]^{=} \\
X^{tG} \ar[r]^{=} \ar[d]_{=} &
\dots \ar[r]^{=} &
X^{tG} \ar[r]^{=} \ar[d] &
X^{tG} \ar[r]^{=} \ar[d] &
\dots \ar[r]^{=} &
X^{tG} \ar[d] \\
X^{tG} \ar[r] &
\dots \ar[r] &
X^{tG}[n,\infty] \ar[r] &
X^{tG}[n{+}1,\infty] \ar[r] &
\dots \ar[r] &
{*}
}
\end{equation}

\begin{lemma} \label{lem_twotowers}
Let $X$ be a $G$-spectrum.  Then
$$
\holim_{n\to-\infty} X^{tG}[-\infty,n] \simeq *
\qquad\text{and}\qquad
\hocolim_{n\to\infty} X^{tG}[n,\infty] \simeq *
$$
so
$$
\holim_{n\to-\infty} X^{tG}[n,\infty] \simeq X^{tG}
\qquad\text{and}\qquad
\hocolim_{n\to\infty} X^{tG}[-\infty,n] \simeq X^{tG} \,.
$$
\end{lemma}

\begin{proof}
For negative $n$, $\tF{n} = D(\tF{m})$ for $m=-n$, and there
is a $G$-equivariant equivalence $\nu \: D(\tF{m}) \wedge Z
\overset{\simeq}\longto \map(\tF{m}, Z)$ for any $G$-spectrum $Z$, since
the finite $G$-CW spectrum $\tF{m}$ is dualizable \cite{LMS}*{III.2.8}.
Hence
\begin{multline*}
\holim_{n\to-\infty} X^{tG}[-\infty,n]
= \holim_{n\to-\infty} \, [\tF{n} \wedge \map(EG_+, X)]^G \\
= \holim_{m\to\infty} \, [D(\tF{m}) \wedge \map(EG_+, X)]^G
\simeq \holim_{m\to\infty} \map(\tF{m} \wedge EG_+, X)^G \\
\cong \map(\hocolim_{m\to\infty} \tF{m} \wedge EG_+, X)^G
\simeq \map(\tEG \wedge EG_+, X)^G \,,
\end{multline*}
which is contractible because $\tEG \wedge EG_+$ is $G$-equivariantly
contractible.

For the second claim we use that $\tEG/\tF{n}$ is a free
$G$-CW spectrum.  Indeed, for $n\ge0$, $\tEG/\tF{n} \simeq \Sigma
(EG/EG^{(n-1)})$.  Thus, by the $G$-Whitehead theorem and the Adams
transfer equivalence~\eqref{equ_adamstransfer} we have
\begin{multline*}
\hocolim_{n\to\infty} X^{tG}[n{+}1,\infty]
= \hocolim_{n\to\infty} \, [\tEG/\tF{n} \wedge \map(EG_+, X)]^G \\
\simeq \hocolim_{n\to\infty} \, [\tEG/\tF{n} \wedge X)]^G
\simeq \hocolim_{n\to\infty} \, [\tEG/\tF{n} \wedge i_* i^* X)]^G \\
\simeq \hocolim_{n\to\infty} \, (\tEG/\tF{n} \wedge i^* X)/G
\cong (\hocolim_{n\to\infty} \tEG/\tF{n} \wedge i^* X)/G \,,
\end{multline*}
which is contractible since $\hocolim_{n\to\infty} \, (\tEG/\tF{n})$
is $G$-equivariantly contractible.
The remaining claims follow, since the homotopy limit of a
fiber sequence is a fiber sequence, and the homotopy colimit
of a cofiber sequence is a cofiber sequence.
\end{proof}

Hereafter we abbreviate $X^{tG}[n, \infty]$ to $X^{tG}[n]$.  We will refer
to the lower horizontal tower $\{X^{tG}[n]\}_n$ in~\eqref{equ_tatetowers}
as the \emph{Tate tower}.  The following two lemmas should be compared
with the sequences~\eqref{equ_cohomtower} and~\eqref{equ_homtower}.

\begin{lemma} \label{lem_otherlimitsvanish}
Let $X$ be a $G$-spectrum.  Then
$$
\lim_{n\to\infty} H^*(X^{tG}[n]) =
\Rlim_{n\to\infty} H^*(X^{tG}[n]) =
\colim_{n\to\infty} H_*(X^{tG}[n]) = 0 \,.
$$
\end{lemma}

\begin{proof}
In cohomology, we have a Milnor $\lim$-$\Rlim$ short exact
sequence
$$
0 \to \Rlim_n H^{*-1}(X^{tG}[n])
\to H^*(\hocolim_n X^{tG}[n]) \to \lim_n H^*(X^{tG}[n]) \to 0 \,.
$$
By Lemma~\ref{lem_twotowers} the middle term is zero, hence so are
the other two terms.  In homology, we have the isomorphism
$$
\colim_{n\to\infty} H_*(X^{tG}[n])
\cong H_*(\hocolim_{n\to\infty} X^{tG}[n]) \,.
$$
By the same lemma the right hand side is zero.
\end{proof}

\begin{lemma}
Suppose that $X$ is bounded below and of finite type over $\F_p$.
Then each spectrum $X^{tG}[n]$ is bounded below and of finite type
over $\F_p$.  Hence
$$
\Rlim_{n\to-\infty} H_*(X^{tG}[n]) = 0 \,.
$$
\end{lemma}

\begin{proof}
Let $X^{tG}[n,m] = [\tF{m}/\tF{n-1} \wedge \map(EG_+, X)]^G$.  For $m
\ge n$ there is a cofiber sequence
$$
X^{tG}[n,m{-}1] \to X^{tG}[n,m] \to \bigvee \Sigma^m X \,,
$$
with one copy of $\Sigma^m X$ in the wedge sum for each of the finitely
many $G$-free $m$-cells in $\tEG$.  Since the connectivity of $\Sigma^m
X$ grows to infinity with $m$, the first claim of the lemma follows by
induction on $m$.  The derived limit of any tower of finite groups is
zero, which gives the second conclusion.
\end{proof}

We use the lower horizontal tower $\{X^{tG}[n]\}_n$
in~\eqref{equ_tatetowers} to define the continuous (co-)homology
of $X^{tG}$, as in Definition~\ref{dfn_continuous} and
Corollary~\ref{cor_HcY}.

\begin{dfn}
Let $G$ be a finite group and $X$ a $G$-spectrum whose underlying
non-equivariant spectrum is bounded below and of finite type over $\F_p$.
By the \emph{continuous cohomology} of $X^{tG}$ we mean the $\A$-module
$$
H_c^*(X^{tG}) = \colim_{n\to-\infty} H^*(X^{tG}[n]) \,.
$$
By the \emph{continuous homology} of $X^{tG}$ we mean the complete
$\A_*$-comodule
$$
H^c_*(X^{tG}) = \lim_{n\to-\infty} H_*(X^{tG}[n]) \,.
$$
There is a natural isomorphism $\Hom(H_c^*(X^{tG}), \F_p) \cong
H^c_*(X^{tG})$ of complete $\A_*$-comodules, as well as a natural
isomorphism $\Hom^c(H^c_*(X^{tG}), \F_p) \cong H_c^*(X^{tG})$ of
$\A$-modules.
\end{dfn}

\begin{remark}
Different $G$-CW structures on $EG$ will give rise to different
Greenlees filtrations and Tate towers, but by cellular approximation
any two choices give pro-isomorphic towers in homology.  The continuous
homology, as a complete $\A_*$-comodule, is therefore independent of
the choice.  Likewise for continuous cohomology.
\end{remark}

We end this subsection by giving a reformulation of the Tate construction,
known as Warwick duality \cite{G1994}*{\textsection 4}, which will be
used in \textsection \ref{subsec_topsinger} when making a topological
model for the Singer construction.

\begin{prop}[\cite{GM1995}*{2.6}] \label{prop_tate2}
There is a natural chain of equivalences
$$
\Sigma \map(\tEG, EG_+ \wedge X)^G \simeq [\tEG \wedge \map(EG_+, X)]^G
	= X^{tG} \,.
$$
\end{prop}

\begin{proof}
We have a commutative diagram of $G$-spectra
\begin{equation} \label{equ_abs}
\xymatrix{
EG_+ \wedge X \ar[r]^-{\map(c,1\wedge1)} & \map(EG_+, EG_+\wedge X) \\
EG_+ \wedge EG_+ \wedge X \ar[u]^{c\wedge1\wedge1}_\simeq \ar[r]
	\ar[d]_{1\wedge\map(c,c\wedge1)}^\simeq
& \map(EG_+, EG_+\wedge X) \ar[u]_{=} \ar[d]^{\map(1,c\wedge1)}_\simeq \\
EG_+ \wedge \map(EG_+, X) \ar[r]^-{c\wedge1} & \map(EG_+,X) \,.
}
\end{equation}
The maps labeled $\simeq$ are $G$-equivalences.  The proposition
follows by taking horizontal (homotopy) cofibers and fixed points.
\end{proof}

We now strengthen this to a statement about towers.

\begin{lemma} \label{lem_tatecomparison}
For each integer $m$ there is a natural chain of equivalences
$$
\Sigma \map(\tF{m}, EG_+ \wedge X)^G \simeq [\tEG/\tF{-m}
	\wedge \map(EG_+, X)]^G = X^{tG}[1{-}m]
$$
connecting the tower
$$
\Sigma \map(\tEG, EG_+ \wedge X)^G \to \dots \to
\Sigma \map(\tF{m+1}, EG_+ \wedge X)^G \to
\Sigma \map(\tF{m}, EG_+ \wedge X)^G \to \dots
$$
to the Tate tower
$$
X^{tG} \to \dots \to X^{tG}[-m] \to X^{tG}[1{-}m] \to \dots \,.
$$
\end{lemma}

\begin{proof}
We give the proof for the more interesting case $m\ge0$, leaving the
case $m<0$ to the reader.
Let $i_m \: EG^{(m-1)}_+ \to EG_+$ be the inclusion, let $c_m \: EG^{(m-1)}_+
\to S^0$ be the restricted collapse map, and let $\delta_m \: EG^{(m-1)}_+
\to EG^{(m-1)}_+ \wedge EG_+$ be the diagonal.
We have a commutative diagram of $G$-spectra
\begin{equation} \label{equ_notsougly}
\xymatrix{
\map(EG_+, EG_+ \wedge X) \ar[rr]^-{\map(i_m,1\wedge1)}
&& \map(EG^{(m-1)}_+, EG_+ \wedge X) \\
\map(EG_+, EG_+ \wedge X) \ar[u]^{=}
	\ar[rr]^-{\map(c_m\wedge1,1\wedge1)}
	\ar[d]_{\map(1,c\wedge1)}^\simeq
&& \map(EG^{(m-1)}_+ \wedge EG_+, EG_+ \wedge X)
	\ar[u]_{\map(\delta_m,1\wedge1)}^\simeq
	\ar[d]^{\map(1\wedge1,c\wedge1)}_\simeq \\
\map(EG_+, X) \ar[rr]^-{\map(c_m,1)}
&& \map(EG^{(m-1)}_+, \map(EG_+, X)) \\
\map(EG_+, X) \ar[rr]^-{Dc_m\wedge1} \ar[u]^{=}
&& D(EG^{(m-1)}_+) \wedge \map(EG_+, X) \ar[u]_{\nu}^\simeq
	\,.
}
\end{equation}
On the right hand side, the middle vertical map makes use of the
identification $\map(EG^{(m-1)}_+ \wedge EG_+, X) \cong \map(EG^{(m-1)}_+,
\map(EG_+, X))$, while the lower vertical map $\nu$ is an equivalence
because $EG^{(m-1)}_+$ is dualizable.

The left hand side of~\eqref{equ_notsougly} matches the right hand side
of~\eqref{equ_abs}.  Combining these two diagrams, and taking horizontal
(homotopy) cofibers, we get a chain of $G$-equivalences connecting the
cofiber of
$$
\map(c_m,1\wedge1) \: EG_+ \wedge X \to \map(EG^{(m-1)}_+, EG_+ \wedge X)
$$
to the cofiber of
$$
(Dc_m \circ c)\wedge1 \: EG_+ \wedge \map(EG_+, X) \to D(EG^{(m-1)}_+)
	\wedge \map(EG_+, X) \,.
$$
The cofiber in the upper row is $G$-equivalent to $\Sigma \map(\tF{m},
EG_+ \wedge X)$, since $\tF{m}$ is the mapping cone of $c_m$.
The cofiber in the lower row is $G$-equivalent to $\tEG/\tF{-m}
\wedge \map(EG_+, X)$, because of the following commutative diagram with
horizontal and vertical cofiber sequences:
$$
\xymatrix{
&& D(\tF{m}) \ar[r]^{=} \ar[d] & \tF{-m} \ar[d] \\
EG_+ \ar[rr]^-c \ar[d]_{=} && S \ar[r] \ar[d]^{Dc_m} & \tEG \ar[d] \\
EG_+ \ar[rr]^-{Dc_m \circ c} && D(EG^{(m-1)}_+) \ar[r] & \tEG/\tF{-m} \\
}
$$
The lemma follows by passage to $G$-fixed points.  It is clear that
these equivalences are compatible for varying $m\ge0$.
\end{proof}

\begin{cor}
The continuous (co-)homology of $X^{tG}$ may be computed from the tower
$$
X^{tG} \to \dots \to
\Sigma \map(\tF{m+1}, EG_+ \wedge X)^G \to
\Sigma \map(\tF{m}, EG_+ \wedge X)^G \to \dots
$$
as
$$
H_c^*(X^{tG}) \cong \colim_{m\to\infty}
	\Sigma H^*(\map(\tF{m}, EG_+ \wedge X)^G)
$$
and
$$
H^c_*(X^{tG}) \cong \lim_{m\to\infty}
	\Sigma H_*(\map(\tF{m}, EG_+ \wedge X)^G) \,.
$$
\end{cor}

\subsection{The (co-)homological Tate spectral sequences}
Let $G$ be a finite group and $X$ a $G$-spectrum.  The cofiber sequence
$\tF{s}/\tF{s-1} \to \tEG/\tF{s-1} \to \tEG/\tF{s}$ induces a cofiber
sequence
$$
[\tF{s}/\tF{s-1} \wedge \map(EG_+, X)]^G \to
	X^{tG}[s] \overset{i}\longto X^{tG}[s{+}1]
$$
for every integer~$s$.  The left hand term is equivalent to
$$
[\tF{s}/\tF{s-1} \wedge X]^G \simeq (\tF{s}/\tF{s-1} \wedge i^* X)/G
$$
since $\tF{s}/\tF{s-1}$ is $G$-free.

Applying cohomology, we get an exact couple of $\A$-modules
\begin{equation} \label{equ_cohomtatecouple}
\xymatrix{
A^{s+1,*} \ar[r]^-i & A^{s,*} \ar[d] \\
& \hatE_1^{s,*} \ar@{-->}[ul]
}
\end{equation}
with
$$
A^{s,t} = H^{s+t}(X^{tG}[s])
\qquad\text{and}\qquad
\hatE_1^{s,t} = H^{s+t}((\tF{s}/\tF{s-1} \wedge i^* X)/G) \,.
$$
By Lemma~\ref{lem_otherlimitsvanish}, $\lim_s A^s = \Rlim_s A^s = 0$, so
this spectral sequence converges conditionally to the colimit
$H_c^*(X^{tG})$, in the first sense of \cite{B1999}*{5.10}.
Applying homology instead, we get an exact couple of algebraic
$\A_*$-comodules
\begin{equation} \label{equ_homtatecouple}
\xymatrix{
A_{s,*} \ar[r]^-i & A_{s+1,*} \ar@{-->}[dl] \\
\hatE^1_{s,*} \ar[u]
}
\end{equation}
with
$$
A_{s,t} = H_{s+t}(X^{tG}[s])
\qquad\text{and}\qquad
\hatE^1_{s,t} = H_{s+t}((\tF{s}/\tF{s-1} \wedge i^* X)/G) \,.
$$
By Lemma~\ref{lem_otherlimitsvanish}, $\colim_s A_s = 0$, so
this spectral sequence converges conditionally to the limit
$H^c_*(X^{tG})$, in the second sense of \cite{B1999}*{5.10}.

We can rewrite the $\hatE^1$-term as
$$
\hatE^1_{s,t} \cong H_s(\tF{s}/\tF{s-1}) \otimes_{\F_pG} H_t(X)
= P_{s-1} \otimes_{\F_pG} H_t(X) \,,
$$
and the $d^1$-differential is induced by the differential in
the complete resolution $(P_*, d_*)$, so
$$
\hatE^2_{s,t} \cong \tH_{s-1}(G; H_t(X)) \cong 
	\tH^{-s}(G; H_t(X)) \,.
$$
Dually, the $\hatE_1$-term is
$$
\hatE_1^{s,t} \cong \Hom(P_{s-1} \otimes_{\F_pG} H_t(X), \F_p)
\cong \Hom_{\F_pG}(P_{s-1}, H^t(X))
$$
and
$$
\hatE_2^{s,t} \cong \tH^{s-1}(G; H^t(X)) \cong \tH_{-s}(G; H^t(X)) \,.
$$

\begin{dfn}
Let $G$ be a finite group and $X$ a $G$-spectrum.  The \emph{cohomological
Tate spectral sequence} of $X$ is the conditionally convergent spectral
sequence
$$
\hatE_2^{s,t} = \tH_{-s}(G; H^t(X)) \Longrightarrow H_c^{s+t}(X)
$$
associated with the exact couple of
$\A$-modules~\eqref{equ_cohomtatecouple}.  Dually, the \emph{homological
Tate spectral sequence} of $X$ is the conditionally convergent spectral
sequence
$$
\hatE^2_{s,t} = \tH^{-s}(G; H_t(X)) \Longrightarrow H^c_{s+t}(X)
$$
associated with the exact couple of
$\A_*$-comodules~\eqref{equ_homtatecouple}.
\end{dfn}

\begin{remark}
There is a natural isomorphism $\hatE_r^{s,t} \cong \Hom(\hatE^r_{s,t},
\F_p)$ for all finite $r$, $s$ and $t$, so that the $d_r$-differential
$d_r^{s,t} \: \hatE_r^{s,t} \to \hatE_r^{s+r,t-r+1}$ is the linear dual
of the $d^r$-differential $d^r_{s+r,t-r+1} \: \hatE^r_{s+r,t-r+1} \to
\hatE^r_{s,t}$.  In this sense the cohomological Tate spectral sequence
is dual to the homological Tate spectral sequence.
\end{remark}

To get strong convergence, we need $X$ to be bounded below in the
cohomological case, and that $X$ is bounded below and of finite type
over $\F_p$ in the homological case.

\begin{prop} \label{prop_cohomological}
Let $G$ be a finite group and $X$ a $G$-spectrum.  Assume that $X$ is
bounded below.  Then $X^{tG}$ is the homotopy inverse limit of a tower
$\{X^{tG}[s]\}_s$ of bounded below spectra, and the cohomological Tate
spectral sequence
$$
\hatE_2^{s,t}(X) = \tH_{-s}(G; H^t(X)) \Longrightarrow H^{s+t}_c(X^{tG})
$$
converges strongly to the continuous cohomology of $X^{tG}$ as an
$\A$-module.
\end{prop}

\begin{proof}
When $H^*(X)$ is bounded below, the cohomological Tate spectral sequence
has exiting differentials in the sense of Boardman, so the spectral
sequence is automatically strongly convergent by \cite{B1999}*{6.1}.
In other words, the filtration of $H^{s+t}_c(X^{tG})$ by
the sub $\A$-modules
$$
F^s H^*_c(X^{tG}) = \im ( H^*(X^{tG}[s]) \to H^*_c(X^{tG}) )
$$
is exhaustive, complete and Hausdorff, and there are $\A$-module
isomorphisms
$$
F^s H^*_c(X^{tG}) / F^{s+1} H^*_c(X^{tG}) \cong \hatE_\infty^{s,*} \,.
$$
\end{proof}

\begin{prop} \label{prop_homological}
Let $G$ be a finite group and $X$ a $G$-spectrum.  Assume that $X$
is bounded below and of finite type over~$\F_p$.  Then $X^{tG}$ is the
homotopy inverse limit of a tower $\{X^{tG}[s]\}_s$ of bounded below
spectra of finite type over~$\F_p$, and the homological Tate spectral
sequence
$$
\hatE^2_{s,t}(X) = \tH^{-s}(G; H_t(X)) \Longrightarrow H_{s+t}^c(X^{tG})
$$
converges strongly to the continuous homology of $X^{tG}$ as a
complete $\A_*$-comodule.
\end{prop}

\begin{proof}
When $H_*(X)$ is bounded below, the homological Tate spectral sequence
has entering differentials in the sense of Boardman.  The derived limit
$R\hatE^\infty_{*,*} = \Rlim_r \hatE^r_{*,*}$ vanishes since $H_*(X)$,
and thus $\hatE^2_{*,*}$, is finite in each (bi-)degree.  Hence the
spectral sequence is strongly convergent by \cite{B1999}*{7.4}.
In other words, the filtration of $H^c_{s+t}(X^{tG})$ by the
complete sub $\A_*$-comodules
$$
F_s H^c_*(X^{tG}) = \ker ( H^c_*(X^{tG}) \to H^c_*(X^{tG}[s]) )
$$
is exhaustive, complete and Hausdorff, and there are
algebraic $\A_*$-comodule isomorphisms
$$
F_{s+1} H^c_*(X^{tG}) / F_s H^c_*(X^{tG}) \cong \hatE^\infty_{s,*} \,.
$$
\end{proof}

\subsubsection{Homotopy vs.~Homology}
We used the lower tower in~\eqref{equ_tatetowers}:
\begin{equation} \label{equ_inversetower}
\xymatrix{
X^{tG} \ar[r] & \dots \ar[r] & X^{tG}[n] \ar[r] & X^{tG}[n{+}1] \ar[r]
	& \cdots \ar[r] & {*}
}
\end{equation}
to define our homological Tate spectral sequence, by applying homology
with $\F_p$-coefficients.  When studying the \emph{homotopy groups}
of the Tate construction, it has been customary to apply $\pi_*(-)$
to the upper tower in~\eqref{equ_tatetowers}:
\begin{equation} \label{equ_directtower}
\xymatrix{
{*} \ar[r] & \dots \ar[r] & X^{tG}[-\infty, n{-}1] \ar[r] &
X^{tG}[-\infty, n] \ar[r] & \dots \ar[r] & X^{tG}
}
\end{equation}
Applying a homological functor to these two towers of spectra gives
two different exact couples with isomorphic spectral sequences.  If we are
working with homotopy, we get two spectral sequences converging to the
same groups.  Using~\eqref{equ_directtower} yields a spectral sequence
converging to the colimit
$$
\colim_n \pi_*(X^{tG}[-\infty,n]) \cong \pi_*(X^{tG}) \,,
$$
while using~\eqref{equ_inversetower} yields an isomorphic spectral
sequence converging to the inverse limit
\begin{equation} \label{equ_nc9}
\lim_n \pi_*(X^{tG}[n]) \cong \pi_*(X^{tG}) \,.
\end{equation}
The latter isomorphism assumes that $X$ is bounded below and of suitably
finite type, so that the right derived limit $\Rlim_n \pi_*(X^{tG}[n])
= 0$.  For example, it suffices if $\pi_*(X)$ is of finite type over $\Z$
or $\widehat\Z_p$.

When working with homology with $\F_p$-coefficients, instead of homotopy
groups, the failure of the isomorphism we made use of in~\eqref{equ_nc9}
makes the situation more interesting.  Applying $H_*(-)$ to the tower
\eqref{equ_directtower} will produce a sequence of homology groups whose
inverse limit is not trivial in general.  This means that the associated
spectral sequence will not be conditionally convergent to the direct limit
$$
\colim_nH_*(X^{tG}[-\infty,n]) \cong H_*(X^{tG})\,.
$$
In fact, we have seen that the (isomorphic) homological Tate spectral
sequence, arising from~\eqref{equ_inversetower}, converges strongly to
$$
\lim_nH_*(X^{tG}[n]) = H_*^c(X^{tG})\,,
$$
which is only rarely isomorphic to $H_*(X^{tG})$, since inverse limits
and homology do not commute in general.

We end this discussion by noticing that the continuous homology groups
of the Tate construction on $X$ can be thought of as the homotopy
groups of the Tate construction on $H \wedge X$, where $H = H\F_p$ is the
Eilenberg--Mac\,Lane spectrum of~$\F_p$.  In other words, continuous
homology of $X^{tG}$ is a special case of homotopy.

\begin{prop} \label{prop_homotopyoftate}
For any bounded below $G$-spectrum $X$ of finite type over~$\F_p$
there is a natural isomorphism
$$
\pi_*(H\wedge X)^{tG} \cong H_*^c(X^{tG}) \,,
$$
where $H$ has the trivial $G$-action.
\end{prop}

\begin{proof}
For all integers~$m$ we have
\begin{align*}
(H\wedge X)^{tG}[1{-}m] &\simeq \Sigma\map(\tF{m}, EG_+\wedge H\wedge X)^G \\
	&\simeq H \wedge \Sigma\map(\tF{m}, EG_+\wedge X)^G
	\simeq H \wedge X^{tG}[1{-}m] \,.
\end{align*}
The first and last equivalences follow from
Lemma~\ref{lem_tatecomparison}, while the middle equivalence follows from
the fact that $\tF{m}$ is $G$-equivariantly dualizable.  Thus, we have
that $\pi_*(H\wedge X)^{tG}[n]\cong H_*(X^{tG}[n])$ for all integers~$n$.
For a general $G$-spectrum $X$, we then have the following surjective
maps:
\begin{align} \label{equ_sfinx}
\pi_* (H\wedge X)^{tG} &\to \lim_{n\to-\infty} \pi_* (H\wedge X)^{tG}[n]\\
&\cong \lim_{n\to-\infty} H_*(X^{tG}[n]) = H_*^c(X^{tG}) \,. \notag
\end{align}
Since $X$ was assumed to be bounded below and of finite type over~$\F_p$,
the groups in the first inverse limit system are all of finite type
over $\F_p$, so their $\Rlim$ vanishes and the map in~\eqref{equ_sfinx} is
an isomorphism.
\end{proof}

The previous proposition and discussion tells us that the continuous
homology of $X^{tG}$ can be considered both as the direct limit
$\colim_n\pi_*(H\wedge X)^{tG}[-\infty, n]$ or as the inverse limit
$\lim_n\pi_*(H\wedge X)^{tG}[n,\infty]$.  In both cases, the filtration
of the two groups given by their defining towers are the same.

\subsection{Multiplicative structure} \label{sec_multstructure}
We now discuss how the treatment in \cite{HM2003}*{\textsection 4.3}
of multiplicative structure in the homotopical Tate spectral sequence
carries over to the homological Tate spectral sequence.

Let $X$ be a bounded below $G$-equivariant ring spectrum of finite
type over~$\F_p$.  We assume that the unit map $\eta \: S \to X$ and
the multiplication map $\mu\: X \wedge X \to X$ are equivariant with
respect to the trivial $G$-action on $S$ and the diagonal $G$-action on
$X \wedge X$.  By \cite{GM1995}*{3.5}, the homotopy cartesian square
\begin{equation} \label{equ_ringsquare}
\xymatrix{
X^G \ar[r]^-R \ar[d]_{\Gamma} & [\tEG \wedge X]^G \ar[d]^{\hat\Gamma} \\
X^{hG} \ar[r]^-{R^h} & X^{tG}
}
\end{equation}
in~\eqref{equ_fund1} is a diagram of ring spectra and ring spectrum maps.
Up to homotopy there is a unique $G$-equivalence $f \: \tEG \wedge
\tEG \overset{\simeq}\to \tEG$, compatible with the inclusion $S^0 \to
\tEG$ and the homeomorphism $S^0 \wedge S^0 \cong S^0$.  Using $f$,
the multiplication map on $X^{tG}$ is given by the composition
$$
\xymatrix{
[\tEG \wedge \map(EG_+, X)]^G \wedge [\tEG\wedge \map(EG_+, X)]^G \ar[d] \\
[\tEG \wedge \tEG \wedge \map(EG_+ \wedge EG_+, X \wedge X)]^G
	\ar[d]^{f \wedge \map(\Delta, \mu)} \\
[\tEG \wedge \map(EG_+, X)]^G \,.
}
$$
The other multiplication maps arise by replacing $\tEG$, $EG_+$ or both
by $S^0$.  The unit map to $X^G$ is adjoint to the $G$-map $S \to X$, since
$S$ has the trivial $G$-action.  The other unit maps arise by composition
with the maps in~\eqref{equ_ringsquare}.

The non-equivariant ring spectrum structure on the Eilenberg--Mac\,Lane
spectrum $H$ makes $H \wedge X$ a naively $G$-equivariant ring spectrum,
so $(H \wedge X)^{tG}$ is a ring spectrum.  The induced graded
ring structure on $\pi_* (H \wedge X)^{tG}$ then gives a graded ring
structure on the continuous homology $H^c_*(X^{tG})$, by the isomorphism
of Proposition~\ref{prop_homotopyoftate}.

\begin{prop} \label{prop_homological2}
Let $G$ be a finite group and $X$ a $G$-equivariant ring spectrum.
Assume that $X$ is bounded below and of finite type over~$\F_p$.
Then the homological Tate spectral sequence
$$
\hatE^2_{s,t} = \tH^{-s}(G; H_t(X))
\Longrightarrow H^c_{s+t}(X^{tG})
$$
is a strongly convergent $\A_*$-comodule algebra spectral sequence,
whose product at the $\hatE^2$-term is given by the cup product in
Tate cohomology and the Pontryagin product on $H_*(X)$.
\end{prop}

\begin{proof}
Using the Greenlees filtration, we may filter $(H\wedge X)^{tG}$
by the tower~\eqref{equ_directtower}.  This produces a homotopical Tate
spectral sequence with $\hatE^2$-term
$$
  \hatE^2_{s,t} = \tH^{-s}(G; \pi_{t}(H\wedge X)) = \tH^{-s}(G; H_t(X)) \,,
$$
converging strongly to the homotopy $\pi_{s+t}(H\wedge X)^{tG} \cong
H^c_{s+t}(X^{tG})$.  By the proof of Proposition~\ref{prop_homotopyoftate}
it is additively isomorphic to the homological Tate spectral sequence
of Proposition~\ref{prop_homological}.
By \cite{HM2003}*{4.3.5}, it is also an algebra
spectral sequence, with differentials being derivations with respect to
the product.  The $\hatE^\infty$-term is the associated graded of
the multiplicative filtration of $\pi_{s+t}(H\wedge X)^{tG}$ given by
the images
$$
\im(\pi_* (H\wedge X)^{tG}[-\infty,s] \to \pi_* (H\wedge X)^{tG})
\,.
$$
\end{proof}

\section{The topological Singer construction} \label{sec_singer}

\subsection{Realizing the Singer construction as continuous cohomology}
As observed by Miller, and presented by Bruner, May, McClure and
Steinberger in \cite{BMMS}*{\textsection II.5}, there is a particular
inverse system of spectra whose continuous cohomology realizes the Singer
construction in the version related to the symmetric group $\Sigma_p$.
We go through the adjustments needed to realize the version of the
Singer construction related to the subgroup $C_p$ generated by the cyclic
permutation $(1~2~\cdots~p)$.

Let $B$ be a non-equivariant spectrum that is bounded below and of finite
type over~$\F_p$.  For each subgroup $G \subseteq \Sigma_p$
there is an extended power construction \cite{BMMS}*{\textsection I.5}
$$
D_G(B) = EG \ltimes_G B^{(p)} \,,
$$
well defined in the stable homotopy category.  Here $B^{(p)}$ denotes
the external $p$-th smash power of $B$, and $G$ permutes the $p$ copies
of $B$.  It follows from \cite{BMMS}*{I.2.4} that $D_G(B)$ is bounded
below and of finite type over~$\F_p$.  More precisely, we have the
following calculation.

\begin{lemma}[\cite{BMMS}*{I.2.3}] \label{lem_grouphomology}
There is a natural isomorphism
$$
H_*(D_G(B))\cong H_*(G; H_*(B)^{\otimes p}) \,,
$$
where $G$ permutes the $p$ copies of $H_*(B)$.
\end{lemma}

The $p$-fold diagonal map $S^1 \to S^1 \wedge \dots \wedge S^1 \cong
S^p$ induces maps $\Sigma D_G(B) \to D_G(\Sigma B)$ as
in \cite{BMMS}*{\textsection II.3}.  Applied to desuspensions of $B$,
these assemble to an inverse system
\begin{equation} \label{equ_inverseextendedpowers}
\dots \longto \Sigma^{n+1} D_G(\Sigma^{-n-1} B)
\overset{\Sigma^n\Delta}\longto \Sigma^n D_G(\Sigma^{-n} B) \longto
\dots \overset{\Delta}\longto D_G(B)
\end{equation}
in the stable homotopy category.  This is a tower of bounded below
spectra of finite type over~$\F_p$, so it makes sense to talk about its
associated continuous cohomology.

We now follow \cite{BMMS}*{\textsection II.5}, but focus on the case $G =
C_p$ instead of $G = \Sigma_p$.  There is an additive isomorphism
$$
H_*(D_{C_p}(B)) \cong
  \F_p\{ e_0 \otimes \alpha_1 \otimes \dots \otimes \alpha_p \}
  \oplus \F_p\{ e_j \otimes \alpha^{\otimes p} \}
$$
where the $\alpha_i$ and $\alpha$ range over a basis for $H_*(B)$, the
$\alpha_i$ are not all equal, and only one representative is taken from
each $C_p$-orbit of the tensors $\alpha_1 \otimes \dots \otimes \alpha_p$.
The grading is determined by $\deg(e_j) = j$.
Dually, there is an isomorphism
$$
H^*(D_{C_p}(B)) \cong
  \F_p\{ w_0 \otimes a_1 \otimes \dots \otimes a_p\}
  \oplus \F_p\{ w_j \otimes a^{\otimes p} \}
$$
where the $a_i$ and $a$ range over the dual basis for $H^*(B)$, and
$w_j \otimes a^{\otimes p}$ is dual to $e_j \otimes \alpha^{\otimes p}$
when $a$ is dual to $\alpha$.  It follows that
\begin{equation} \label{equ_cohomsuspextpower}
H^*(\Sigma^n D_{C_p}(\Sigma^{-n} B)) \cong \F_p\{ \Sigma^n w_0
	\otimes \Sigma^{-n} a_1 \otimes \dots \otimes \Sigma^{-n} a_p\}
\oplus \F_p\{ \Sigma^n w_j \otimes (\Sigma^{-n} a)^{\otimes p} \} \,.
\end{equation}
By \cite{BMMS}*{II.5.6}, the map $\Delta$
in~\eqref{equ_inverseextendedpowers} is given in cohomology as
$$
\Delta^*( w_j \otimes a^{\otimes p} ) = (-1)^{j+1} \alpha(q) \cdot
	\Sigma w_{j+p-1} \otimes (\Sigma^{-1} a)^{\otimes p}
$$
where $\deg(a) = q$, $m = (p-1)/2$ and $\alpha(q) = -(-1)^{mq} m!$.
For $p=2$ the numerical coefficient should be read as $1$.
The other classes $w_0 \otimes a_1 \otimes \dots \otimes a_p$ map to zero.
It follows that
\begin{equation} \label{equ_cohosigmandelta}
(\Sigma^n\Delta)^*(\Sigma^n w_j \otimes (\Sigma^{-n} a)^{\otimes p})
= (-1)^{j+1} \alpha(q{-}n) \cdot \Sigma^{n+1} w_{j+p-1} \otimes
(\Sigma^{-n-1} a)^{\otimes p}  \,.
\end{equation}
The action of the Steenrod algebra $\A$ on $H^*(D_{C_p}(B))$ is given
by the Nishida relations.  Together with the explicit formula above
for the maps $(\Sigma^n\Delta)^*$, this determines the direct limit of
cohomology groups as an $\A$-module.

Miller observed that this direct limit can be described in closed
form by the Singer construction on the $\A$-module $H^*(B)$, up to a
single degree shift.  We now give the $C_p$-equivariant extension of
the $\Sigma_p$-equivariant case discussed in \cite{BMMS}*{II.5.1}.

\begin{thm} \label{thm_TopSinger}
For each spectrum $B$ that is bounded below and of finite type
over~$\F_p$, there is a natural isomorphism of $\A$-modules
$$
\omega \: \colim_{n\to\infty} H^*(\Sigma^n D_{C_p}(\Sigma^{-n} B))
\overset{\cong}\longto \Sigma^{-1} R_+(H^*(B)) \,.
$$
\end{thm}

\begin{proof}
For $p=2$ the isomorphism is given by
$$
\omega(\Sigma^n w_{j+n} \otimes (\Sigma^{-n} a)^{\otimes 2})
	= x^{j+q} \otimes a
$$
where $\deg(a) = q$.  For $p$ odd, the isomorphism 
is given by
$$
\omega(\Sigma^n w_{2(r+mn)} \otimes (\Sigma^{-n} a)^{\otimes p})
	= (-1)^{q-n} \nu(q{-}n)^{-1} \cdot y^{r+mq} \otimes a
$$
and
$$
\omega(\Sigma^n w_{2(r+mn)-1} \otimes (\Sigma^{-n} a)^{\otimes p})
	= (-1)^{q} \nu(q{-}n)^{-1} \cdot x y^{r+mq-1} \otimes a \,,
$$
where $\deg(a) = q$, $m = (p-1)/2$ and $\nu(2j+\epsilon) = (-1)^j
(m!)^\epsilon$ for $\epsilon \in \{0, 1\}$.

It follows from~\eqref{equ_cohosigmandelta} and the relation $\alpha(q)
\nu(q-1)^{-1} = \nu(q)^{-1}$ that these homomorphisms are compatible under
$(\Sigma^n\Delta)^*$.  Then it is clear from~\eqref{equ_cohomsuspextpower}
that $\omega$ is an additive isomorphism.  It also commutes
with the Steenrod operations on the extended powers, described in
\cite{BMMS}*{II.5.5}, and the explicitly defined Steenrod operations
on the Singer construction, given in Definition~\ref{dfn_rplus}.
After some computation, the thing to check is that $\alpha(q) (-1)^{q+1}
\nu(q+1)^{-1} = (-1)^q \nu(q)^{-1}$, which follows from the relation
above and $\nu(q+1) = -\nu(q-1)$.
\end{proof}

\subsection{The relationship between the Tate and Singer constructions}
\label{subsec_topsinger}
We now show that the inverse system~\eqref{equ_inverseextendedpowers}
for $G = C_p$ and $B$ a bounded below spectrum can be realized, up to
a single suspension, as the Tate tower in~\eqref{equ_tatetowers} for a
$C_p$-spectrum $X = B^{\wedge p}$.

As a naive $C_p$-spectrum, $B^{\wedge p}$ is equivalent to the $p$-fold
smash product $B \wedge \dots \wedge B$, with the $C_p$-action given by
cyclic permutation of the factors.  The genuinely equivariant definition
of $B^{\wedge p}$ is obtained by specialization from B{\"o}kstedt's
definition in \cite{B1}, \cite{BHM1993} of the topological Hochschild
homology $\THH(B)$ of a symmetric ring spectrum $B$, in the sense
of~\cite{HSS2000}.  Namely, $B^{\wedge p} = \sd_p \THH(B)_0 =
\THH(B)_{p-1}$ equals the $0$-simplices of the $p$-fold edgewise
subdivision of $\THH(B)$, which in turn equals the $(p-1)$-simplices
of $\THH(B)$.

The ring structure on $B$ is only relevant for the simplicial structure on
$\THH(B)$, and is not needed for the formation of its $(p-1)$-simplices.
However, it is necessary to assume that the spectrum $B$ is realized
as a symmetric spectrum.  We now make a review of definitions, in order to
compare the B{\"o}kstedt-style smash powers of symmetric spectra with
the external powers of Lewis--May spectra.

From now on, let $\U$ be the complete $C_p$-universe
$$
\U = \R^\infty \oplus \dots \oplus \R^\infty = (\R^\infty)^p
$$
with $C_p$-action given by cyclic permutation of summands.  The inclusion
$i \: \R^\infty \to \U$ is the diagonal embedding $\Delta \: \R^\infty
\to (\R^\infty)^p$.  Let $\fA = \{\R^n \mid n\ge0\}$ be the sequential
indexing set in $\R^\infty$, and let $\fA^p = \{\R^n \oplus \dots \oplus
\R^n = (\R^n)^p \mid n\ge0\}$ be the associated diagonal indexing set
\cite{LMS}*{\textsection VI.5} in $\U$.  Recall that a prespectrum $D$
is \emph{$\Sigma$-cofibrant} in the sense of \cite{LMS}*{I.8.7}, hence
good in the sense of Hesselholt and Madsen~\cite{HM1997}*{Def.~A.1},
if each structure map $\Sigma^{W-V} D(V) \to D(W)$ is a cofibration for
$V \subseteq W$ in the indexing set.  There is a functorial thickening
$D^\tau$ of prespectra, which produces $\Sigma$-cofibrant, hence good,
prespectra, and there is a natural spacewise equivalence $D^\tau
\to D$, see \cite{HM1997}*{Lem.~A.1}.  All of this works just as well
equivariantly.

Let $B$ be a symmetric spectrum of topological spaces, with $n$-th space
$B_n$ for each $n\ge0$.  Recall that $B$ is $S$-cofibrant in the sense of
\cite{HSS2000}*{5.3.6} if the natural map $\nu_n \: L_n B \to B$ is a
cofibration for each $n$, where the latching space $L_n B$ is the $n$-th
space in the spectrum $B \wedge \bar S$.  Here $\bar S$ is the symmetric
spectrum with $0$-th space $*$ and $n$-th space $S^n$, for $n>0$.
We prefer to follow the terminology in \cite{Schwede}*{III.1.2}
and refer to the $S$-cofibrant symmetric spectra as being \emph{flat}.
Every symmetric spectrum is level equivalent, hence stably
equivalent, to a flat symmetric spectrum, so any spectrum may be modeled
by a flat symmetric spectrum.
Each symmetric spectrum~$B$ has an underlying sequential prespectrum
indexed on $\fA$, with $B(\R^n) = B_n$ equal to the $n$-th space
of $B$.  The structure map $\sigma \: B(\R^{n-1}) \wedge S^1 \to
B(\R^n)$ factors as $B_{n-1} \wedge S^1 \overset{\iota_n}\longto L_n B
\overset{\nu_n}\longto B_n$, where $\iota_n$ is always a cofibration.
Hence the underlying prespectrum of a flat symmetric spectrum is
$\Sigma$-cofibrant.

\begin{dfn} \label{dfn_bsmashb}
Let $B$ be a symmetric spectrum, with $n$-th space $B_n$ for each
$n\ge0$, and let $I$ be B{\"o}kstedt's category of finite sets
$\bfn = \{1, 2, \dots, n\}$ for $n\ge0$ and injective functions.
Let $B^{\wedge p}_{\pre}$ be the $C_p$-equivariant prespectrum
with $V$-th space
$$
B^{\wedge p}_{\pre}(V) = \hocolim_{(\bfn_1, \dots, \bfn_p) \in I^p}
\Map(S^{n_1} \wedge \dots \wedge S^{n_p},
	B_{n_1} \wedge \dots \wedge B_{n_p} \wedge S^V)
$$
for each finite dimensional $V \subset \U$.  Here $C_p$ acts by cyclically
permuting the $\bfn_i$, the $S^{n_i}$ and the $B_{n_i}$, as well as
acting on $S^V$.  Let
$$
B^{\wedge p} = L((B^{\wedge p}_{\pre})^\tau)
$$
be the genuine $C_p$-spectrum in $C_p\S\U$ obtained by spectrification
from the functorial good thickening $(B^{\wedge p}_{\pre})^\tau$ of
this prespectrum.  The natural maps
$$
B^{\wedge p}_{\pre}(V)
\overset{\simeq}\longleftarrow
(B^{\wedge p}_{\pre})^{\tau}(V)
\overset{\simeq}\longrightarrow
B^{\wedge p}(V)
$$
are $C_p$-equivariant equivalences by the proof of
\cite{HM1997}*{Prop.~2.4}.
\end{dfn}

%% Why is thickening needed?  Is $B^{\wedge p}_{\pre}$ not
%% an inclusion prespectrum?

\begin{dfn}
Let $B$ be a prespectrum indexed on $\fA$.  The $p$-fold external
smash product $B^{(p)}$ is the spectrification in $C_p\S\U$ of the
$C_p$-equivariant prespectrum
$$
B^{(p)}_{\pre}((\R^n)^p) = B(\R^n) \wedge \dots \wedge B(\R^n)
= B(\R^n)^{\wedge p}
$$
indexed on $\fA^p$.  When $B$ is $\Sigma$-cofibrant, so is
$B^{(p)}_{\pre}$, hence the $V$-th space of the spectrification is given
by the colimit
$$
B^{(p)}(V) = \colim_{V \subseteq (\R^n)^p}
	\Map(S^{(\R^n)^p - V}, B(\R^n)^{\wedge p}) \,.
$$
Here the colimit runs over the $n \in \N_0$ such that $V \subseteq
(\R^n)^p$, and $(\R^n)^p - V$ denotes the orthogonal complement of $V$
in $(\R^n)^p$.  Suspension by $V$ induces an isomorphism
$$
B^{(p)}(V) \cong \colim_{n \in \N_0}
\Map((S^n)^{\wedge p}, B(\R^n)^{\wedge p} \wedge S^V) \,,
$$
with inverse given by suspension by $(\R^m)^p - V$, followed by $p$
instances of the stabilization $B(\R^n) \wedge S^m \to B(\R^{n+m})$,
for $m$ sufficiently large.
\end{dfn}

We say that a symmetric spectrum $B$ is \emph{convergent} if there are
integers $\lambda(n)$ that grow to infinity with $n$, such that $B_n$
is $((n/2) + \lambda(n))$-connected and the structure map $\Sigma B_n \to
B_{n+1}$ is $(n+\lambda(n))$-connected, for all sufficiently large~$n$.
These hypotheses suffice for the use of B{\"o}kstedt's approximation lemma
in the following proof.  Every symmetric spectrum is stably equivalent
to a convergent one.

\begin{lemma} \label{lem_comparepthpowers}
Let $B$ be a flat, convergent symmetric spectrum.  There is
a natural chain of weak equivalences of naive $C_p$-spectra
$$
i^* B^{(p)} \simeq i^* B^{\wedge p} \,.
$$
\end{lemma}

\begin{proof}
For every finite dimensional $V \subset \R^\infty$ we have a natural
chain of $C_p$-equivariant maps
$$
\xymatrix{
\displaystyle{\colim_{n \in \N_0}}
	\Map((S^n)^{\wedge p}, B(\R^n)^{\wedge p} \wedge S^V) \\
\displaystyle{\hocolim_{n \in \N_0}}
	\Map((S^n)^{\wedge p}, B(\R^n)^{\wedge p} \wedge S^V)
\ar[u]_\simeq \ar[d]^\simeq \\
\displaystyle{\hocolim_{(n_1, \dots, n_p) \in \N_0^p}}
	\Map(S^{n_1} \wedge \dots \wedge S^{n_p},
	B_{n_1} \wedge \dots \wedge B_{n_p} \wedge S^V)
\ar[d]^\simeq \\
\displaystyle{\hocolim_{(\bfn_1, \dots, \bfn_p) \in I^p}}
	\Map(S^{n_1} \wedge \dots \wedge S^{n_p},
	B_{n_1} \wedge \dots \wedge B_{n_p} \wedge S^V)
}
$$
connecting $B^{(p)}(V)$ to $B^{\wedge p}_{\pre}(V)$.  The upper map
is a weak equivalence because $B$ is flat, hence $\Sigma$-cofibrant.
The middle map is a weak homotopy equivalence because the diagonal
$\N_0 \to \N_0^p$ is (co-)final.  The lower map is a weak equivalence
by convergence, the fact that $\N_0^p$ is filtering, and B{\"o}kstedt's
approximation lemma \cite{B1}*{1.5}, see \cite{B2000}*{2.5.1}
for a published proof.  Applying the thickening construction and
spectrifying, we get the desired chain of naively $C_p$-equivariant weak
equivalences.
\end{proof}

\begin{dfn}
For $p=2$, let $\R(1)$ be the sign representation of $C_2$.  For $p$ odd,
let $\C(1)$ be the standard rank~$1$ representation of $C_p \subset S^1$,
and let $\C(i)$ be its $i$-th tensor power.  For all primes $p$, let
$W
\subset \R^p$ be the orthogonal complement of the diagonal copy of $\R$.
Then $W \cong \R(1)$ for $p=2$, and $W \cong \C(1) \oplus \dots \oplus
\C(m)$ for $p$ odd, where $m = (p-1)/2$.  Let $EC_p = S(\infty W)$,
$\tEC_p = S^{\infty W}$, and give $EC_p$ a $C_p$-CW structure so that
$$
\tEC_p^{((p-1)n)} = S^{nW}
$$
for all $n\ge0$.  Then $\tF{(p-1)n} = S^{nW}$ for all integers~$n$.
\end{dfn}

\begin{prop} \label{prop_taterewrite}
Let $B$ be a flat and convergent symmetric spectrum, and give
$EC_p$ a free $C_p$-CW structure as in the definition above.
There is a natural weak equivalence
$$
(B^{\wedge p})_{hC_p} = (EC_{p+} \wedge i^* B^{\wedge p})/C_p
\simeq EC_p \ltimes_{C_p} B^{(p)} = D_{C_p}(B) \,.
$$
More generally, there are weak equivalences
$$
(B^{\wedge p})^{tC_p}[1{-}(p{-}1)n] \simeq \Sigma^{1+n} D_{C_p}(\Sigma^{-n} B)
$$
for all $n\ge0$, which are compatible with the $(p-1)$-fold
composites of maps in the Tate tower~\eqref{equ_tatetowers} for $X =
B^{\wedge p}$
$$
\dots \to (B^{\wedge p})^{tC_p}[1{-}(p{-}1)(n{+}1)] \to
(B^{\wedge p})^{tC_p}[1{-}(p{-}1)n] \to \dots \,,
$$
and the suspension of the inverse system~\eqref{equ_inverseextendedpowers}
for $G = C_p$
$$
\dots \to \Sigma^{1+n+1} D_{C_p}(\Sigma^{-n-1} B) 
\to \Sigma^{1+n} D_{C_p}(\Sigma^{-n} B) \to \dots \,.
$$
\end{prop}

\begin{proof}
By the untwisting theorem \cite{LMS}*{VI.1.17} and
Lemma~\ref{lem_comparepthpowers} there are
$C_p$-equivariant equivalences
$$
EC_p \ltimes B^{(p)} \simeq EC_{p+} \wedge i^* B^{(p)} \simeq
	EC_{p+} \wedge i^* B^{\wedge p} \,,
$$
since $EC_p$ is $C_p$-free.  The first claim follows by passage
to $C_p$-orbit spectra.

More generally, there are $C_p$-equivariant equivalences
$$
\Sigma^{1+n} EC_p \ltimes (\Sigma^{-n} B)^{(p)} \simeq
\Sigma \map(S^{nW}, EC_p \ltimes B^{(p)}) \simeq
\Sigma \map(S^{nW}, EC_{p+} \wedge i^* B^{\wedge p})
$$
by \cite{LMS}*{VI.1.5}, since $S^n \wedge S^{nW} \cong
(S^n)^{\wedge p}$.  Passing to $C_p$-orbits, and using the Adams transfer
equivalence~\eqref{equ_adamstransfer}, we get the equivalences
$$
\Sigma^{1+n} D_{C_p}(\Sigma^{-n} B) \simeq
\Sigma \map(S^{nW}, EC_{p+} \wedge B^{\wedge p})^{C_p} \,.
$$
The right hand side is a model for $(B^{\wedge p})^{tC_p}[1{-}(p{-}1)n]$,
by Lemma~\ref{lem_tatecomparison}, since $\tF{(p-1)n} = S^{nW}$.  The
stabilization of the left hand side given by $\Delta \: S^1 \to S^p$
is compatible under all of these equivalences with the stabilization
of the right hand side given by the inclusion $S^0 \to S^W$, again by
\cite{LMS}*{VI.1.5}.
\end{proof}

\begin{dfn}
For each symmetric spectrum $B$ let the \emph{topological Singer
construction} on $B$ be the spectrum
$$
R_+(B) = (B^{\wedge p})^{tC_p} \,.
$$
\end{dfn}

The topological Singer construction realizes the algebraic Singer
constructions, in the following sense.

\begin{thm} \label{thm_topmodelsalg}
Let $B$ be a symmetric spectrum that is bounded below and of finite
type over $\F_p$.  There are natural isomorphisms
$$
\omega \: H^*_c(R_+(B)) \overset{\cong}\longto R_+(H^*(B))
$$
and
$$
\omega_* \: R_+(H_*(B)) \overset{\cong}\longto H_*^c(R_+(B))
$$
of $\A$-modules and complete $\A_*$-comodules, respectively.
\end{thm}

\begin{proof}
We may replace $B$ by a stably equivalent flat and convergent symmetric
spectrum, without changing the (co-)homology of $B^{\wedge p}$ and
$R_+(B)$.  By Proposition~\ref{prop_taterewrite} there are $\A$-module
isomorphisms
$$
H^*((B^{\wedge p})^{tC_p}[1{-}(p{-}1)n]) \cong
\Sigma^{1+n} H^*(D_{C_p}(\Sigma^{-n} B))
$$
for each $n$, which by Theorem~\ref{thm_TopSinger}
induce $\A$-module isomorphisms
$$
H^*_c((B^{\wedge p})^{tC_p}) \cong R_+(H^*(B))
$$
after passage to colimits.  The dual $\A_*$-comodule isomorphisms
then induce complete $\A_*$-comodule isomorphisms
$$
H^c_*((B^{\wedge p})^{tC_p}) \cong
	\Hom(R_+(H^*(B)), \F_p) = R_+(H_*(B))
$$
after passage to limits.
\end{proof}

\subsection{The topological Singer $\epsilon$-map}
	\label{subsec_topepsilon}
We now turn to the construction of a stable map $\epsilon_B \: B \to
R_+(B)$ that realizes the Singer homomorphism $\epsilon$ on passage to
cohomology.  The homotopy fixed point property for the $C_p$-equivariant
spectrum $B^{\wedge p}$ will then follow easily.

Let $B$ be a symmetric spectrum that is bounded below and of finite
type over~$\F_p$.  The $C_p$-spectrum $X = B^{\wedge p}$ introduced in
Definition~\ref{dfn_bsmashb} is then also bounded below and of finite
type over~$\F_p$.  We shall make use of parts of the Hesselholt--Madsen
proof that $\THH(B)$ is a \emph{cyclotomic spectrum}.  By the first half
of the proof of \cite{HM1997}*{Prop.~2.1} there is a natural equivalence
$$
\bar s_{C_p} \: [\tEC_p \wedge B^{\wedge p}]^{C_p} \overset{\simeq}\longto
	\Phi^{C_p}(B^{\wedge p}) \,,
$$
where $\Phi^{C_p}(X)$ denotes the \emph{geometric fixed point spectrum}
of $X$.  Furthermore, by the simplicial degree~$k=p-1$ part of the
proof of \cite{HM1997}*{Prop.~2.5} there is a natural equivalence
$$
r'_{C_p} \: \Phi^{C_p}(B^{\wedge p}) \overset{\simeq}\longto B
\,.
$$
If $B$ is a ring spectrum, both of these equivalences are
ring spectrum maps.

\begin{dfn}
Let $\epsilon_B \: B \to R_+(B)$ be the natural stable map given by
the composite
$$
B \overset{\simeq}\longleftarrow \Phi^{C_p}(B^{\wedge p})
  \overset{\simeq}\longleftarrow [\tEC_p \wedge B^{\wedge p}]^{C_p}
  \overset{\hat\Gamma}\longto (B^{\wedge p})^{tC_p} = R_+(B) \,.
$$
If $B$ is a ring spectrum then $\epsilon_B$ is a ring spectrum map.
\end{dfn}

With this notation we can rewrite the homotopy cartesian square
in~\eqref{equ_fund1} for $X = B^{\wedge p}$ as follows:
\begin{equation} \label{equ_fundbsmashb}
\xymatrix{
(B^{\wedge p})^{C_p} \ar[r]^-R \ar[d]_{\Gamma} &
B \ar[d]^{\epsilon_B} \\
(B^{\wedge p})^{hC_p} \ar[r]^-{R^h} & R_+(B)
}
\end{equation}

We thank M.~B{\"o}kstedt for a helpful discussion on the following
two results.

\begin{lemma} \label{lem_suspensionepsilon}
The stable map $\epsilon_B \: B \to R_+(B)$ commutes with suspension,
in the sense that $\epsilon_{\Sigma B} = \Sigma \epsilon_B$.
\end{lemma}

\begin{proof}
Consider the commutative diagram:
$$
\xymatrix{
\Sigma B \ar[d]_{=} & \Sigma [\tEC_p \wedge B^{\wedge p}]^{C_p}
  \ar[l]_-{\simeq} \ar[r]^-{\Sigma\hat\Gamma} \ar[d]^{\Delta} &
  \Sigma R_+(B) \ar[d]^{\Delta} \\
\Sigma B & [\tEC_p \wedge (\Sigma B)^{\wedge p}]^{C_p}
  \ar[l]_-{\simeq} \ar[r]^-{\hat\Gamma} & R_+(\Sigma B)
}
$$
The vertical maps labeled $\Delta$ are induced by the diagonal inclusion
$S^1 \to S^p$, which on the right hand side is the same map as was used
in the interpretation (Proposition~\ref{prop_taterewrite}) of $R_+(B)$
as the inverse system of suspended extended power constructions.
Hence these maps are weak equivalences.
\end{proof}

\begin{prop} \label{prop_epsilonB}
Let $B$ be a bounded below spectrum of finite type over~$\F_p$.
Then the homomorphism
$$
(\epsilon_B)^* \: H_c^*(R_+(B)) \to H^*(B)
$$
induced on continuous cohomology by the spectrum map $\epsilon_B \:
B \to R_+(B)$ is equal to Singer's homomorphism
$$
\epsilon_{H^*(B)} \: R_+(H^*(B)) \to H^*(B)
$$
associated to the $\A$-module $H^*(B)$.
\end{prop}

\begin{proof}
By Corollary~\ref{cor_uniquehom} there is a unique $\A$-module
homomorphism $g_B \: H^*(B) \to H^*(B)$ that makes the square
$$
\xymatrix{
R_+(H^*(B)) \ar[d]_{\epsilon_{H^*(B)}}
  & H_c^*(R_+(B)) \ar[d]^{(\epsilon_B)^*} \ar[l]_-{\omega}^-{\cong} \\
H^*(B) \ar[r]^-{g_B} & H^*(B)
}
$$
commute.  We must show that $g_B$ equals the identity.

First consider the case $B = H$.  The homological Tate spectral sequence
$$
\hatE^2_{*,*}(H) = \tH^{-*}(C_p; H_*(H)^{\otimes p})
\Longrightarrow H^c_*(R_+(H))
$$
is an algebra spectral sequence (Proposition~\ref{prop_homological2}), and
$\epsilon_H \: H \to R_+(H)$ is a ring spectrum map, so the image of $1
\in H_0(H)$ under $(\epsilon_H)_*$ is represented by $1 \otimes 1^{\otimes
p}$ in $\tH^0(C_p; H_0(H)^{\otimes p})$.  Hence $(\epsilon_H)^*$ maps
the dual class represented by $1 \otimes 1^{\otimes p}$ in $\tH_0(C_p;
H^0(H)^{\otimes p})$ to $1 \in H^0(H)$.  Now $\omega(1 \otimes 1^{\otimes
p}) = \Sigma xy^{-1} \otimes 1 \in R_+(\A)$ and $\epsilon_{\A}(\Sigma
xy^{-1} \otimes 1) = 1 \in \A$, so $g_H(1) = 1$.  (We replace $xy^{-1}$
by $x^{-1}$ for $p=2$.)  Since $g_H$ is an $\A$-module homomorphism,
it must be equal to the identity $\A = H^*(H)$.

The case $B = \Sigma^n H$ then follows by Lemma~\ref{lem_suspensionepsilon}.

In the general case any element in $H^n(B)$ is represented by a map $f
\: B \to \Sigma^n H$, which induces an $\A$-module homomorphism $f^* \:
\Sigma^n \A \to H^*(B)$.  By naturality of the isomorphism $\omega$,
the Singer homomorphism $\epsilon$, and the spectrum map $\epsilon$,
we get a diagram
$$
\xymatrix{
R_+(\Sigma^n \A) \ar[ddd]_{\epsilon_{\Sigma^n \A}}
  \ar[dr]_-{R_+(f^*)}
  &&& H_c^*(R_+(\Sigma^n H)) \ar[ddd]^{(\epsilon_{\Sigma^n H})^*}
  \ar[lll]_-{\omega}^-{\cong} \ar[dl]^-{R_+(f)^*} \\
& R_+(H^*(B)) \ar[d]_{\epsilon_{H^*(B)}}
  & H_c^*(R_+(B)) \ar[d]^{(\epsilon_B)^*} \ar[l]_-{\omega}^-{\cong} \\
& H^*(B) \ar[r]_-{g_B} & H^*(B) \\
\Sigma^n \A \ar[rrr]^-{=} \ar[ur]^-{f^*}
  &&& H^*(\Sigma^n H) \ar[ul]_-{f^*}
}
$$
where the left hand, upper and right hand trapezoids all commute.  The
inner square commutes by construction, and the outer square commutes by
the case $B = \Sigma^n H$.  Since $\epsilon_{\Sigma^n \A}$ is surjective,
it follows that the lower trapezoid also commutes.  Hence $g_B$ equals
the identity on the class $f^*(\Sigma^n 1) \in H^n(B)$.  Since $n$ and
$f$ were arbitrary, this proves that $g_B$ equals the identity on all
of $H^*(B)$.
\end{proof}

The following theorem generalizes the Segal conjecture for $C_p$.

\begin{thm} \label{thm_fixedpoints}
Let $B$ be a bounded below spectrum of finite type over~$\F_p$.  Then the
natural maps
$$
\epsilon_B \: B \to R_+(B) = (B^{\wedge p})^{tC_p}
$$
and
$$
\Gamma\: (B^{\wedge p})^{C_p} \to (B^{\wedge p})^{hC_p}
$$
are $p$-adic equivalences of spectra.
\end{thm}

\begin{proof}
The map $\epsilon_B$ induces a map of spectral sequences
$$
E_2^{*,*}(B) = \Ext_{\A}^{*,*}(H^*(B), \F_p) \to
\Ext_{\A}^{*,*}(H_c^*(R_+(B)), \F_p) = E_2^{*,*}(R_+(B))
$$
where the first is the Adams spectral sequence of $B$, and the second is
the inverse limit of Adams spectral sequences associated to the Tate tower
$\{(B^{\wedge p})^{tC_p}[n]\}_n$, as in Proposition~\ref{prop_CMPss}.
The map converges strongly to the homomorphism
$$
\pi_*(\epsilon_B)\sphat_p \: \pi_*(B\sphat_p) \to \pi_*(R_+(B)\sphat_p)
\,.
$$
By Propositions~\ref{prop_epsilonB} and~\ref{prop_extiso} and
Theorem~\ref{thm_extiso}, the map of $E_2$-terms is an isomorphism, hence
so is the map of $E_\infty$-terms and of the abutments.  In other words,
$\epsilon_B$ is a $p$-adic equivalence.  The corresponding assertion
for $\Gamma$ follows immediately, since diagram~\eqref{equ_fundbsmashb}
is homotopy cartesian.
\end{proof}

\subsection{The Tate spectral sequence for the topological Singer construction}
\label{subsec_tatessforsinger}
We conclude by relating the homological Tate spectral sequence for
$B^{\wedge p}$ to the Tate filtration on the homological Singer
construction for $H_*(B)$, and likewise in cohomology.

\begin{prop}
Let $B$ be a bounded below spectrum of finite type over~$\F_p$.
The homological Tate spectral sequence
$$
\hatE^2_{*,*} = \tH^{-*}(C_p; H_*(B)^{\otimes p})
\Longrightarrow H^c_*((B^{\wedge p})^{tC_p})
$$
converging to $H^c_*(R_+(B)) \cong R_+(H_*(B))$ collapses at the
$\hatE^2$-term.  Hence the $\hatE^2 = \hatE^\infty$-term is given by
$$
\hatE^\infty_{*,*} = P(u, u^{-1}) \otimes
	\F_2\{\alpha^{\otimes 2}\}
$$
for $p=2$, and by
$$
\hatE^\infty_{*,*} = E(u) \otimes P(t, t^{-1}) \otimes
	\F_p\{\alpha^{\otimes p}\}
$$
for $p$ odd.
In each case $\alpha$ runs through an $\F_p$-basis for $H_*(B)$.

For $p=2$ and any $r \in \Z$, $\alpha \in H_q(B)$, the element $u^r
\otimes \alpha \in R_+(H_*(B))$ is represented in the Tate spectral
sequence by
$$
u^{r+q} \otimes \alpha^{\otimes 2} \in \hatE^\infty_{-r-q, 2q} \,.
$$
For $p$ odd and any $i \in \{0,1\}$, $r \in \Z$ and $\alpha \in H_q(B)$,
the element $u^i t^r \otimes \alpha \in R_+(H_*(B))$ is represented in
the Tate spectral sequence by
$$
(-1)^q \nu(q)^{-1} \cdot u^i t^{r+mq} \otimes \alpha^{\otimes p}
	\in \hatE^\infty_{-i-2r-(p-1)q, pq}
\,,
$$
where $m=(p-1)/2$ and $\nu(2j+\epsilon) = (-1)^j (m!)^\epsilon$
for $\epsilon \in \{0,1\}$.
\end{prop}

\begin{proof}
Consider first the case $B=S^q$.  Then the result is trivial for
dimensional reasons.  The $\hatE^2$-term is concentrated in bidegrees $(*,
pq)$, so there is no room for differentials, and there are no extension
problems to be solved.  The formula for the homological Tate spectral
sequence representative for $u^i t^r \otimes \alpha$ will follow
by dualization from the cohomological case, given below.

Let $H$ be a model for the mod~$p$ Eilenberg--Mac\,Lane spectrum as a
commutative symmetric ring spectrum, and form the spectrum $H^{\wedge
p}$ in $C_p\S\U$ as in Definition~\ref{dfn_bsmashb}.  The iterated
multiplication on $H$ then induces a naively $C_p$-equivariant map
$H^{\wedge p} \to H$.
Let $f \: S^q \to H \wedge B$ represent a class $\alpha \in H_q(B)$,
and consider the naively $C_p$-equivariant composite map
$$
f^p \: H\wedge S^{pq} \to H\wedge (H\wedge B)^{\wedge p}
\simeq H\wedge H^{\wedge p}\wedge B^{\wedge p}
\to H \wedge H \wedge B^{\wedge p} \to H\wedge B^{\wedge p} \,.
$$
On homotopy groups it induces the homomorphism $H_*(S^q)^{\otimes
p} \to H_*(B)^{\otimes p}$ that takes $\iota_q^{\otimes p}$ to
$\alpha^{\otimes p}$, where $\iota_q = \Sigma^q 1$ is the fundamental
class in $H_q(S^q)$.

By applying the Tate construction to this map, we get a map of spectra
$$
(f^p)^{tC_p} \: (H \wedge S^{pq})^{tC_p} \to (H \wedge B^{\wedge p})^{tC_p}
$$
and an associated map of homotopical Tate spectral sequences, converging
to an $\F_p$-linear map
\begin{equation} \label{equ_univmap}
H^c_*(R_+(S^q)) \to H^c_*(R_+(B))
\end{equation}
by Proposition~\ref{prop_homotopyoftate}.  For $p$ odd, this map is
given at the level of $\hatE^2$-terms as sending $u^i t^r \otimes
\iota_q^{\otimes p}$ to $u^i t^r \otimes \alpha^{\otimes p}$, and
similarly for $p=2$.  The statement of the proposition then follows
by naturality.
\end{proof}

Note that the $\F_p$-linear map~\eqref{equ_univmap} is not a homomorphism
of $\A_*$-comodules, because $f^p$ was formed using the product on $H$.

\begin{prop}
Let $B$ be a bounded below spectrum of finite type over~$\F_p$.
The cohomological Tate spectral sequence
$$
\hatE_2^{*,*} = \tH_{-*}(C_p; H^*(B)^{\otimes p})
\Longrightarrow H_c^*((B^{\wedge p})^{tC_p})
$$
converging to $H_c^*(R_+(B)) \cong R_+(H^*(B))$ collapses at the
$\hatE_2$-term, so that
$$
\hatE_\infty^{*,*} = \Sigma P(x, x^{-1}) \otimes
	\F_2\{a^{\otimes 2}\}
$$
for $p=2$, and
$$
\hatE_\infty^{*,*} = \Sigma E(x) \otimes P(y, y^{-1}) \otimes
	\F_p\{a^{\otimes p}\}
$$
for $p$ odd.
In each case $a$ runs through an $\F_p$-basis for $H^*(B)$.

For $p=2$ and any $r \in \Z$, $a \in H^q(B)$, the element $\Sigma x^r
\otimes \alpha \in R_+(H^*(B))$ is represented in the Tate spectral
sequence by
$$
\Sigma x^{r-q} \otimes a^{\otimes 2} \in \hatE_\infty^{1+r-q, 2q} \,.
$$
For $p$ odd and any $i \in \{0,1\}$, $r \in \Z$ and $a \in H^q(B)$,
the element $\Sigma x^i y^r \otimes a \in R_+(H^*(B))$ is represented in
the Tate spectral sequence by
$$
(-1)^q \nu(q) \cdot \Sigma x^i y^{r-mq} \otimes a^{\otimes p}
	\in \hatE_\infty^{1+i+2r-(p-1)q, pq}
\,.
$$
\end{prop}

\begin{proof}
This follows by dualization from the homological case.  In the special
case $B = S^q$, $\Sigma x^i y^r \otimes a^{\otimes p} \in H_c^*(R_+(B))
\cong \tH^{-*}(C_p; H_*(S^q)^{\otimes p})$ is represented by $(-1)^q
\nu(q)^{-1} \cdot \Sigma x^i y^{r+mq} \otimes a \in R_+(H_*(S^q))$, by the
explicit isomorphism given in the proof of Theorem~\ref{thm_TopSinger}.
The formula for the cohomological Tate spectral sequence representative
follows.
\end{proof}

\begin{cor} \label{cor_comparefiltrations}
The Tate filtration
$$
\{F^nR_+(H^*(B))\}_n
$$
of the Singer construction $R_+(H^*(B))$ corresponds, under the
isomorphism $R_+(H^*(B)) \cong H_c^*(R_+(B))$, to the Boardman
filtration
$$
\{F^nH_c^*(R_+(B))\}_n
$$
of $H_c^*(R_+(B))$.
\end{cor}

\begin{proof}
For each integer~$n$, the Boardman filtration $F^nH_c^*(R_+(B))$
equals the image of $H^*((B^{\wedge p})^{tC_p}[n])$ in $H^*((B^{\wedge
p})^{tC_p})$, which is the part of $H^*((B^{\wedge p})^{tC_p})$
represented in filtrations $\ge n$ at the $\hatE_\infty$-term.  This
corresponds to the part of the Singer construction $H^*(R_+(B))$ spanned
by the monomials $\Sigma x^r \otimes a$ with $1+r-q \ge n$ for $p=2$, and
by the monomials $\Sigma x^i y^r \otimes a$ with $1+i+2r-(p-1)q \ge n$
for $p$ odd, which precisely equals the $n$-th term $F^nR_+(H^*(B))$ of
the Tate filtration, as defined in \textsection \ref{subsec_homalgsinger}.
\end{proof}

\end{document}